\documentclass{amsart}
\usepackage{amssymb}
\usepackage{amscd}
\usepackage{amsfonts}
\usepackage{latexsym}
\usepackage[all]{xy}
\usepackage{verbatim}
\usepackage{mathenv}

\newtheorem{theorem}{Theorem}[section]
\newtheorem{lemma}[theorem]{Lemma}
\newtheorem{proposition}[theorem]{Proposition}
\newtheorem{corollary}[theorem]{Corollary}
\theoremstyle{definition}
\newtheorem{definition}[theorem]{Definition}

\newtheorem{remark}[theorem]{Remark}

\DeclareMathOperator{\id}{id}

\DeclareMathOperator{\FPdim}{FPdim}

\DeclareMathOperator{\Fun}{Fun}

\renewcommand{\Vec}{\operatorname{Vec}}

\DeclareMathOperator{\Hom}{Hom}

\DeclareMathOperator{\Aut}{Aut}

\DeclareMathOperator{\Rep}{Rep}

\newcommand{\rev}{\text{rev}}

\renewcommand{\mod}{\operatorname{mod}}

\DeclareMathOperator{\ev}{ev}
\DeclareMathOperator{\coev}{coev}

\DeclareMathOperator{\Real}{Re}

\newcommand{\C}{\mathcal{C}}
\newcommand{\T}{\mathcal{T}}

\newcommand{\D}{\mathcal{D}}
\newcommand{\E}{\mathcal{E}}

\newcommand{\Z}{\mathcal{Z}}

\renewcommand{\L}{\mathcal{L}}
\newcommand{\M}{\mathcal{M}}

\renewcommand{\O}{\mathcal{O}}

\newcommand{\be}{\mathbf{1}}

\renewcommand{\be}{\mathbf{1}}

\newcommand{\bt}{\boxtimes}
\newcommand{\ot}{\otimes}
\newcommand{\TY}{\mathcal{T}\mathcal{Y}}
\newcommand{\bZ}{\mathbb{Z}}
\newcommand{\legendre}[2]{#1 \overwithdelims () #2}

\begin{document}
\title{Centers of graded fusion categories}

\author{Shlomo Gelaki}
\address{S.G.: Department of Mathematics, Technion-Israel Institute of
Technology, Haifa 32000, Israel.}
\email{gelaki@math.technion.ac.il}

\author{Deepak Naidu}
\address{D.N.: Department of Mathematics, Texas A\&M University,
College Station, TX 77843, USA.}
\email{dnaidu@math.tamu.edu}

\author{Dmitri Nikshych}
\address{D.N.: Department of Mathematics and Statistics,
University of New Hampshire,  Durham, NH 03824, USA.}
\email{nikshych@math.unh.edu}

\begin{abstract}
Let $\C$ be a fusion category faithfully graded by a finite group~$G$
and let $\D$ be the trivial component of this grading.
The center $\Z(\C)$  of $\C$ is shown to be canonically equivalent to a $G$-equivariantization
of the relative center $\Z_\D(\C)$.
We use this result to
obtain a criterion for $\C$ to be  group-theoretical and apply it
to Tambara-Yamagami fusion categories.   We also find
several new series of
modular categories by analyzing the centers of Tambara-Yamagami
categories.   Finally, we prove a general result about existence of zeroes in $S$-matrices
of weakly integral modular categories.
\end{abstract}

\maketitle

\setcounter{tocdepth}{1}
\tableofcontents
\newpage

\section{Introduction}

Throughout the paper we work over an algebraically closed field $k$
of characteristic $0$. All categories considered in this paper are
finite, Abelian, semisimple, and $k$-linear.
We freely use the language and basic theory of fusion categories, module
categories over them, braided categories, and Frobenius-Perron dimensions \cite{BK, O, ENO1}.

Let $G$ be a finite group. A fusion category $\C$ is {\em $G$-graded}
if there is a decomposition
\[
\C =\bigoplus_{g\in G}\, \C_g
\]
of $\C$ into a direct sum of full Abelian subcategories such that
the tensor product of $\C$
maps $\C_g\times \C_h$ to $\C_{gh}$, for all $g,h\in G$.
A {\em $G$-extension} of a fusion category $\D$ is a $G$-graded
fusion category $\C$  whose trivial component $\C_e$,
where $e$ is the identity of $G$, is equivalent to $\D$.

Gradings and extensions play an important role in the study
and classification of fusion categories.  E.g., {\em nilpotent}
fusion categories  (i.e., those categories that can be obtained from
the trivial category by a sequence of groups extensions) were studied
in \cite{GN}.  It was proved in \cite{ENO1} that every fusion category
of prime power dimension is nilpotent.  Group-theoretical properties
of such categories were studied in \cite{DGNO}.  Recently, fusion
categories of dimension $p^nq^m$, where $p,\,q$ are primes,
were shown to be Morita equivalent to nilpotent categories \cite{ENO3}.

The main goal of this paper is to describe the center $\Z(\C)$
of a $G$-graded  fusion category $\C$ in terms
of its trivial component $\D$ (Theorem~\ref{ZCG =ZDC})
and apply this description to the study of structural properties of $\C$
and construction of new examples of modular categories.

The organization of the paper is as follows. In Section~\ref{sect 2}
we recall some basic notions, results, and examples of fusion
categories, notably the notions of the relative center of a bimodule
category \cite{Ma}, group action on a fusion category and crossed
product \cite{Ta2}, equivariantization and de-equivariantization
theory \cite{AG, Br, G, Ki, Mu1, DGNO}, and braided $G$-crossed
fusion categories \cite{Tu1, Tu2}.

In Section~\ref{main section} we study the center $\Z(\C)$ of a $G$-graded fusion
category $\C$. We show  that if $\D$ is the trivial component of $\C$,
then the relative center $\Z_\D(\C)$ has a canonical structure of
a braided $G$-crossed category and  there is an equivalence of
braided fusion categories $\Z_\D(\C)^G \cong \Z(\C)$
(Theorem~\ref{ZCG =ZDC}). Thus, the structure of $\Z(\C)$ can be understood
in terms of a smaller and more transparent category $\Z_\D(\C)$.
In particular, there is a canonical braided action of $G$
on  $\Z(\D)$\footnote{This action is studied in detail in \cite{ENO3}.}.
In Corollary~\ref{graded gt criterion} we use this action to prove
that $\C$ is group-theoretical if and only if $\Z(\D)$
contains a $G$-stable Lagrangian subcategory.
As an illustration, we describe the center of a crossed product
fusion category $\C= \D \rtimes G$.

We apply the above results in Section~\ref{TY section}
to the study  of Tambara-Yamagami categories \cite{TY}.
We obtain a convenient description of the
centers of such categories as equivariantizations
and compute their modular data, i.e., $S$- and $T$-matrices.
This computation was previously done by Izumi in \cite{I}
using different techniques.  We  establish a criterion for a Tambara-Yamagami
category  to be group-theoretical (Theorem~\ref{thm:when TY is gt}).
We also extend the construction of non group-theoretical semisimple
Hopf algebras from Tambara-Yamagami categories given in \cite{Ni}.

In Section~\ref{modquadr}  we construct a series of new modular categories as factors of
the centers of Tambara-Yamagami categories.  Namely, one associates a pair
of such categories $\E(q,\,\pm)$ with any non-degenerate quadratic form $q$
on an Abelian group $A$ of odd order.  The categories $\E(q,\,\pm)$ have dimension
$4|A|$. They are group-theoretical if and only if $A$ contains a Lagrangian
subgroup with respect to $q$.  We compute the $S$-
and $T$- matrices  of $\E(q,\,\pm)$ and  write down several small examples
explicitly.

Section~\ref{sect 6} is independent from the rest of the paper and contains
a general result about existence of zeroes in $S$-matrices of
weakly integral modular categories (Theorem~\ref{Burnside again}).
This is a categorical analogue of a classical result of Burnside in character theory.

\par
{\bf Acknowledgments.}  We are grateful to P.~Etingof, M.~M\"uger, and V.~Ostrik for useful
discussions. Part of this work was done while the first
author was on Sabbatical in the Departments of Mathematics at the
University of New Hampshire and MIT; he is grateful for their warm
hospitality. The research of S.~Gelaki was partially
supported by the Israel Science Foundation (grant No. 125/05).
The research of D.~Nikshych was supported by
the NSA grant H98230-07-1-0081 and the NSF grant DMS-0800545.

\section{Preliminaries}
\label{sect 2}

Below we recall  several constructions and results used in the sequel.
\subsection{Dual fusion categories  and Morita equivalence}
Let $\C$ be a fusion category
and let $\M$ be an  indecomposable right $\C$-module
category $\M$. The category  $\C^*_\M$ of $\C$-module
endofunctors of $\M$ is a fusion category, called the dual
of $\C$ with respect to $\M$ (see \cite{ENO1, O}).

Following \cite{Mu3}, we say that two fusion categories $\C$ and
$\D$ are {\em  Morita equivalent} if $\D$ is equivalent to $\C^*_\M$, for some indecomposable
right $\C$-module category $\M$.
A fusion category is said to be {\em pointed} if all its simple objects are
invertible (any such category is equivalent to the category $\Vec_G^\omega$
of vector spaces graded by a finite group $G$ with the associativity constraint
given by a $3$-cocylce $\omega \in Z^3(G,\, k^\times)$).
A fusion category is called {\em group-theoretical} if it is
Morita equivalent to a pointed fusion category.
See \cite{O, ENO1, Ni} for details of the theory of group-theoretical categories.

\subsection{The center of a bimodule category and the relative center of a fusion category}
\label{relative center}

Let $\C$ be a fusion category with unit object $\be$ and
associativity constraint  $\alpha_{X,Y,Z}:(X\ot Y)\ot Z \xrightarrow{\sim} X\ot (Y \ot Z)$
and let $\M$ be a $\C$-bimodule
category.

\begin{definition}
\label{center of a module}
The {\em center} of $\M$ is the category $\Z_\C(\M)$ of $\C$-bimodule functors
from $\C$ to $\M$.
\end{definition}

Explicitly, the objects of $\Z_\C(\M)$ are pairs $(M,\, \gamma)$, where $M$ is an object of $\M$ and
\begin{equation}
\label{gamma}
\gamma = \{ \gamma_{X} : X\ot M \xrightarrow{\sim}  M \ot X \}_{ X \in \C}
\end{equation}
is a natural family of isomorphisms making the following diagram
commutative:
\begin{equation}
\label{central object}
\xymatrix{
& X \ot (M \ot Y) \ar[rr]^{\alpha_{X,M,Y}^{-1}} & &
(X \ot M) \ot Y  \ar[dr]^{\gamma_{X}} & \\
X\ot (Y\ot M) \ar[ur]^{\gamma_{Y}}
\ar[dr]_{\alpha_{X,Y,M}^{-1}} & &  &  &(M \ot X) \ot Y \\
& (X\ot Y)\ot M \ar[rr]_{\gamma_{X\ot Y}} & & M\ot (X\ot Y)
\ar[ur]_{\alpha_{M,X,Y,}^{-1}} & }
\end{equation}
where $\alpha$'s denote the associativity constraints in $\M$.

Indeed, a $\C$-bimodule functor $F: \C \to \M$ is completely determined
by the pair $(F(\be),\, \{\gamma_X\}_{X\in \C})$, where $\gamma = \{\gamma_X\}_{X\in \C}$
is the collection of isomorphisms
\[
\gamma_X : X \ot F(\be) \xrightarrow{\sim}  F(X) \xrightarrow{\sim}   F(\be) \ot X
\]
coming from the $\C$-bimodule structure on $F$.

We will  call the natural family of isomorphisms \eqref{gamma} the {\em central structure} of
an object $ X \in \Z_\C(\M)$.

\begin{remark}
\begin{enumerate}
\item[(i)] The definition of the center of a bimodule category is parallel to that
of the center of a bimodule over a ring.
\item[(ii)] We will often suppress the central structure  while working
with objects of $\Z_\C(\M)$ and refer to  $(M,\,\gamma)$ simply as $M$.
\item[(iii)] $\Z_\C(\M)$ is a semisimple Abelian category. It has an obvious canonical structure
of a $\Z(\C)$-module category, where $\Z(\C)$ is the center of $\C$
(see e.g., \cite[Section XIII.4]{K}  for the definition of
$\Z(\C)$).

\end{enumerate}
\end{remark}

Here is an important special case of the above construction. 
Let $\C$ be a fusion category and let $\D \subset \C$ be a fusion
subcategory. Then $\C$ is a $\D$-bimodule category. We will call
$\Z_\D(\C)$ the {\em relative center} of  $\C$.

\begin{remark}
The aforementioned construction of relative center is a special case
of a more general construction considered by Majid in \cite{Ma}
(see Definition 3.2 and Theorem 3.3 of \cite{Ma}).
\end{remark}

It is easy to see
that $\Z_\D(\C)$ is a tensor category with tensor product defined as
follows. If $(X,\, \gamma)$ and $(X',\, \gamma')$ are objects in
$\Z_\D(\C)$ then
\[
(X,\, \gamma) \otimes (X',\, \gamma') := (X\ot X',\,
\tilde\gamma),
\]
where $\tilde\gamma_V: V \ot
(X\ot X') \xrightarrow{\sim} (X\ot X')\ot V ,\, V\in \D$,
is defined by the following diagram:
\begin{equation}
\label{centrproduct}
\xymatrix{
V\ot (X\ot X') \ar[d]_{\tilde\gamma_V} \ar[rr]^{\alpha_{V,X, X'}^{-1}} &&
(V\ot X) \ot X'  \ar[rr]^{\gamma_V} &&
(X\ot V) \ot X' \ar[d]^{\alpha_{X, V, X'}} \\
(X\ot X') \ot V  &&
X\ot (X'\ot V)  \ar[ll]_{\alpha_{X,X',V}^{-1}} &&
X\ot (V \ot X') \ar[ll]_{\gamma'_V}.
}
\end{equation}
The unit object of $\Z_\D(\C)$ is $(\mathbf{1},\,\id)$. The dual of $(X,\,\gamma)$
is $(X^*,\,\overline{\gamma})$, where $\overline{\gamma}_V:= (\gamma_{{}^*V})^*$.

\begin{remark} Let $\C$ and $\D$ be as above.
\begin{enumerate}
\item[(i)] $\Z_\D(\C)$ is dual to the fusion category $\D\bt \C^\text{rev}$
(where $\C^\text{rev}$ is the fusion category obtained from $\C$ by
reversing the tensor product and $\bt$ is Deligne's tensor product
of fusion categories)
with respect to its module category $\C$, where $\D$ and $\C^\text{rev}$ act
on $\C$ via the right  and left multiplication respectively. In
particular, $\Z_\D(\C)$ is a fusion category.
\item[(ii)]  $\FPdim(\Z_\D(\C)) = \FPdim(\C)\FPdim(\D)$, where $\FPdim$
denotes the Frobenius-Perron dimension of a category.
\item[(iii)] $\Z_\C(\C)$ coincides with the center $\Z(\C)$ of  $\C$. This
category has a canonical braiding given by
\begin{equation}
\label{braiding}
c_{(X, \gamma),\,(X',\gamma')} = \gamma_{X'} :
 (X,\, \gamma) \otimes (X',\, \gamma') \xrightarrow{\sim}
(X',\, \gamma') \otimes (X,\, \gamma).
\end{equation}
\item[(iv)] There is an obvious forgetful tensor functor:
\begin{equation}
\label{ZC to ZDC} \Z(\C) \mapsto \Z_\D(\C): (X,\,\gamma) \mapsto
(X,\,\gamma|_\D).
\end{equation}
\end{enumerate}
\end{remark}


\subsection{Centralizers in braided fusion categories}
\label{centralizers etc}

Let $\C$ be a braided fusion cate\-gory with braiding $c$.
Two objects $X$ and $Y$ of $\C$ are said to {\em centralize}
each other \cite{Mu2} if $c_{Y,X}c_{X,Y}=\id_{X\ot Y}$.

For any fusion subcategory $\D \subseteq \C$ its {\em centralizer}
$\D'$ is the full fusion subcategory of $\C$ consisting of all objects $X \in \C$
which centralizes every object in $\D$. The category $\C$ is said to be
{\em non-degenerate} if $\C'=\Vec$.   In this case one has $\D''=\D$ \cite{Mu2}.
If $\C$ is a pre-modular category,
i.e., has a spherical structure, then it is non-degenerate if and only if it is
modular.

A braided fusion category $\E$ is called {\em Tannakian} if it is
equi\-valent to the representation category $\Rep(G)$ of a finite
group $G$ as a braided fusion category. Here $\Rep(G)$ is considered
with its standard symmetric braiding. The group $G$ is defined by
$\E$ up to an isomorphism \cite{D}.

A fusion subcategory $\L$ of a braided fusion category is called {\em Lagrangian} if
it is Tannakian and $\L =\L'$.

\begin{theorem}[\cite{DGNO}]
\label{gt DGNO}
 A fusion category $\C$ is group-theoretical if and only if
$\Z(\C)$ contains a Lagrangian subcategory.
\end{theorem}


\subsection{Group actions on fusion categories and equivariantization}
\label{group actions}

Let $G$ be a finite group, and let ${\underline G}$ denote the
monoidal category whose objects are elements of $G$, morphisms are
identities, and the tensor product is given by the multiplication in
$G$. Recall that an action of $G$ on  a fusion category $\C$ is a
monoidal functor ${\underline G} \to \text{Aut}_\ot(\C): g \mapsto
T_g$. For any $g,h\in G$, let $\gamma_{g,h}$ be the isomorphism
$T_g\circ T_h\simeq T_{gh}$ that defines the monoidal structure on
the functor ${\underline G} \to \text{Aut}_\ot(\C)$.

\begin{definition}
\label{Gequiv object} A {\em $G$-equivariant object\,} in $\C$ is a
pair $(X,\{u_g\}_{g\in G})$ consisting of an object $X$ of $\C$
together with a collection of isomorphisms
 $u_g: T_g(X)\simeq X,\, g\in G$, such that the diagram
\begin{equation*}
\label{equivariantX}
\xymatrix{T_g(T_h(X))\ar[rr]^{T_g(u_h)} \ar[d]_{\gamma_{g,h}(X)
}&&T_g(X)\ar[d]^{u_g}\\ T_{gh}(X)\ar[rr]^{u_{gh}}&&X}
\end{equation*}
commutes for all $g,h\in G$.
One defines morphisms of equivariant objects to be morphisms in $\C$ commuting with $u_g,\; g\in G$.
\end{definition}

Equivariant objects in $\C$ form a fusion category, called the  {\em equivariantization}
of $\C$ and denoted by $\C^G$,  see \cite{Ta2, AG, G}.
One has  $\FPdim(\C^G) = |G| \FPdim(\C)$.

There is another fusion category that comes from an action of $G$ on $\C$.
It is the {\em crossed product} category $\C \rtimes G$ defined as follows, see \cite{Ta2, Ni}.
As an Abelian category, $\C \rtimes  G: = \C \bt \Vec_G$, where $\Vec_G$ denotes
the fusion category of $G$-graded vector spaces.
The tensor product in $\C \rtimes G$  is given by
\begin{equation}
\label{crossed prod}
(X \bt g) \ot (Y \bt h): = (X \ot T_g(Y)) \bt gh,\qquad X,Y \in \C,\quad g,h\in G.
\end{equation}
The unit object is $\be \bt e$ and the associativity and unit constraints come from those of $\C$.
Clearly, $\C \rtimes G$ is faithfully $G$-graded with the trivial component $\C$.

It was explained in \cite{Ni} that $\C$ is a right $\C\rtimes
G$-module category via
\[
Y \ot (X\bt g):= T_{g^{-1}}(Y \ot X)
\]
and the corresponding dual category $(\C\rtimes G)_\C^*$ is equivalent to $\C^G$.
It follows from \cite{Mu3} that there is an equivalence of braided fusion categories
\[
\Z(\C\rtimes G) \cong \Z(\C^G).
\]

Let $G$ be a finite group. For any conjugacy class $K$ of $G$ fix
a representative $a_K\in K$. Let $G_K$ denote the centralizer of $a_K$ in $G$.

\begin{proposition}
\label{simples in CG}
Let $\C=\bigoplus_{g\in G}\, \C_g$ be a $G$-graded fusion category with an
action $g \mapsto T_g$ of
$G$ on $\C$ such that $T_g$ carries $\C_h$ to $\C_{ghg^{-1}}$.
Let $H :=\{ g\in G \mid \C_g \neq 0 \}$.
There is a bijection between the set of isomorphism classes of simple
objects of $\C^G$  and pairs $(K,\, X)$, where $K \subset H$
is a conjugacy class of $G$ and $X$ is a simple $G_K$-equivariant
object of $\C_{a_K}$.
\end{proposition}
\begin{proof}
A simple $G$-equivariant object of $\C$ must be supported on a single
conjugacy class $K$.  Let $Y=\oplus_{g\in K}\, Y_g$ be such an object.
Then $Y_{a_K}$ is a simple  $G_K$-equivariant object.

Conversely, given a $G_K$-equivariant object $X$ in $\C_{a_K}$ let
\[
Y =\bigoplus_{h}\, T_h(X),
\]
where the summation is taken over the set of representatives
of cosets of $G_K$ in $G$. It is easy to see that $Y$ acquires the structure
of a simple $G$-equivariant object.

Clearly, the above constructions are inverses of each other.
\end{proof}

\begin{remark}
The Frobenius-Perron dimension of the simple object corresponding to
a pair $(K,\, X)$
in Proposition~\ref{simples in CG} is $|K| \FPdim(X)$. \\
\end{remark}

\subsection{De-equivariantization of fusion categories}
\label{deeq}

Let $\C$ be a fusion category.  Let $\E =\Rep(G)$
be a Tannakian category along with a braided tensor functor $\E \to \Z(\C)$
such that the composition $\E \to \Z(\C)\to \C$ (where the second arrow is the forgetful functor)
is fully faithful.  The following construction was introduced
by Brugui\`{e}res \cite{Br} and M\"uger \cite{Mu1}.
Let $A:=\mbox{Fun}(G)$ be the algebra of functions on
$G$. It is a commutative algebra in $\E$, hence, its image is a commutative algebra in $\Z(\C)$.
This fact allows to view
the category $\C_G$ of $A$-modules in $\C$ as a
fusion category, called {\em de-equivariantization}  of $\C$.
There is a canonical surjective tensor functor
\begin{equation}
\label{de-eq functor}
F: \C \to \C_G : X \mapsto A\ot X.
\end{equation}
It was explained in \cite{Mu1, DGNO} that the group $G$ acts on $\C_G$
by tensor auto\-equivalences (this action comes from the action of $G$
on $A$ by right translations). Furthermore, there is  a bijection
between subcategories of $\C$ containing the image of $\E=\Rep(G)$ and $G$-stable
subcategories of $\C_G$. This bijection preserves Tannakian subcategories.

The procedures of equivariantization and  de-equivariantization
are inverses of each other, i.e., there are canonical equivalences
$(\C_G)^G \cong \C$ and $(\C^G)_G \cong \C$.

In particular, the above construction applies when $\C$ is a braided
fusion category containing a Tannakian subcategory $\E=\Rep(G)$.
In this case  the braiding of $\C$  gives rise to an additional structure on
the de-equivariantization functor \eqref{de-eq functor}.
Namely, there is natural family of isomorphisms
\begin{equation}
\label{Central Func}
X \ot F(Y)  \xrightarrow{\sim} F(Y)\ot X, \qquad  X\in \C_G,\, Y\in \C,
\end{equation}
satisfying obvious compatibility conditions.
In other words, $F$ can be factored through a braided functor $\C \to \Z(\C_G)$,
i.e., $F$ is a {\em central} functor.

If  $\E \subset \C'$ then $\C_G$ is a braided fusion category
with the braiding inherited from that of $\C$. If
$\E=\C'$, the category $\C_G$  is non-degenerate (in the presence
of a spherical structure this category is called
the {\em modularization} of $\C$ by $\E$ \cite{Br, Mu1}).

\begin{remark}
The category $\C_G$ is not braided in general. However it does have
an additional structure, namely it is a {\em braided $G$-crossed
fusion category}. See Section~\ref{G-equivariant vector bundles}
below for details.
\end{remark}
\subsection{Braided $G$-crossed categories}
\label{G-equivariant vector bundles}

Let $G$ be a finite group.
Kirillov Jr.\ \cite{Ki} and M\"uger \cite{Mu4} found a  description
of all braided fusion categories $\D$ containing $\Rep(G)$. Namely,
they showed that the datum of a braided fusion category $\D$
containing $\Rep(G)$ is equivalent to the datum of a braided
$G$-crossed category $\C$, see Theorem~\ref{KM thm}.
The notion  of a braided $G$-crossed category  is due to Turaev \cite{Tu1, Tu2}
and is recalled below.

\begin{definition}
\label{crossed} A {\em braided $G$-crossed fusion category} is a
fusion category $\C$ equip\-ped with  the following structures:
\begin{enumerate}
\item[(i)] a (not necessarily faithful)
grading $\C=\bigoplus_{g\in G}\C_g$,
\item[(ii)]  an action $g\mapsto T_g$ of $G$ on $\C$
such that $T_g(\C_h)\subset \C_{ghg^{-1}}$,
\item[(iii)] a natural
collection of isomorphisms, called the {\em $G$-braiding}:
\begin{equation}
\label{crosssed braiding} c_{X,Y}: X\ot Y\simeq T_g(Y)\ot X, \qquad
X\in \C_g,\, g\in G \mbox{ and } Y\in \C.
\end{equation}
\end{enumerate}
Let $\gamma_{g,h} : T_gT_h\xrightarrow{\sim} T_{gh}$ denote the tensor
structure of the functor $g\mapsto T_g$ and let $\mu_g$ denote the tensor
structure of $T_g$.

The above structures are required to satisfy
the following compatibility conditions:
\begin{enumerate}
\item[(a)] the diagram
\begin{equation}
\label{G-hex0}
\xymatrix @C=0.6in @R=0.45in{
T_g(X) \ot T_g(Y) \ar[rr]^{c_{T_g(X), T_g(Y)}} & &
T_{ghg^{-1}}(T_g(Y)) \ot T_g(X)  \ar[d]^{(\gamma_{ghg^{-1},g})_Y \ot \id_{T_g(X)}} & \\
T_g(X\ot Y) \ar[u]^{(\mu_g)_{X,Y}^{-1}}
\ar[d]_{T_g(c_{X,Y})} &  &  T_{gh}(Y) \ot T_g(X) \\
T_g(T_h(Y) \ot X) \ar[rr]_{(\mu_g)_{T_g(Y),X}^{-1}} & &
T_g(T_h(Y)) \ot T_g(X) \ar[u]_{(\gamma_{g,h})_Y \ot \id_{T_g(X)}},
}
\end{equation}
commutes for all $g,h \in G$ and objects $X\in \C_h,\, Y\in \C$,
\item[(b)] the diagram
\begin{equation}
\label{G-hex1}
\xymatrix{
& (X\ot Y)\ot Z \ar[dl]_{\alpha_{X,Y,Z}} \ar[dr]^{c_{X,Y}\ot \id_Z} & \\
X\ot (Y\ot Z) \ar[d]_{c_{X,Y\ot Z}} &  & (T_g(Y)\ot X) \ot Z \ar[d]^{\alpha_{T_g(Y),X,Z}} \\
T_g(Y \ot Z) \ot X \ar[d]_{(\mu_g)_{Y,Z}^{-1}\ot \id_X} & & T_g(Y) \ot (X \ot Z) \ar[d]^{\id_{T_g(Y)}\ot c_{X,Z}} \\
(T_g(Y) \ot T_g(Z)) \ot X  \ar[rr]^{\alpha_{T_g(Y),T_g(Z), X}} & & T_g(Y) \ot (T_g(Z) \ot X)
}
\end{equation}
commutes for all $g \in G$ and objects $X \in \C_g, Y,Z \in \C,$ and
\item[(c)] the diagram
\begin{equation}
\label{G-hex2}
\xymatrix{
& X\ot (Y\ot Z) \ar[dr]^{\id_X \ot c_{Y,Z}} & \\
(X\ot Y)\ot Z  \ar[ur]^{\alpha_{X,Y,Z}} &  & X\ot (T_h(Z)\ot Y) \ar[d]^{\alpha^{-1}_{X, T_h(Z),Y}} \\
T_{gh}(Z) \ot (X \ot Y)  \ar[u]^{c_{X\ot Y, Z}^{-1}}  & & (X\ot T_h(Z))\ot Y \ar[d]^{c_{X, T_h(Z)}\ot \id_Y} \\
T_gT_h(Z) \ot (X  \ot Y)  \ar[u]^{(\gamma_{g,h})_{Z}\ot \id_{X\ot Y}}  \ar[rr]^{\alpha^{-1}_{T_gT_h(Z),X, Y}} & & (T_gT_h(Z) \ot X)  \ot Y.
}
\end{equation}
commutes for all $g,h \in G$ and objects $X\in \C_g,\,Y\in \C_h, Z \in \C$.
\end{enumerate}
\end{definition}

\begin{remark}
\label{BGC rems}
The trivial component $\C_e$ of a braided $G$-crossed fusion category $\C$ is
a braided fusion category with the action of $G$ by braided autoequivalences.
This can be seen by taking $X,\,Y\in \C_e$ in diagrams \eqref{G-hex0} -- \eqref{G-hex2}.
\end{remark}

\begin{theorem} [\cite{Ki, Mu4}]
\label{KM thm} The equivariantization and de-equivariantization
constructions establish  a bijection between the set of equivalence
classes of $G$-crossed braided fusion categories and the set of
equivalence classes of braided fusion categories containing
$\Rep(G)$ as a symmetric fusion subcategory.
\end{theorem}

We shall now sketch the proof of this theorem.
An alternative approach is given in \cite{DGNO}. 

Suppose $\C$ is a braided $G$-crossed fusion category. We define a
braiding $\tilde{c}$ on its equivariantization $\C^G$ as follows.

Let $(X,\, \{u_g\}_{g\in G})$ and $(Y,\, \{v_g\}_{g\in G})$ be
objects in $\C^G$. Let $X =\oplus_{g\in G}\,X_g$ be a decomposition
of $X$ with respect to the grading of $\C$.  Define an isomorphism
\begin{equation}
\label{braiding from equiv}
\tilde c_{X,Y}: X\ot Y =\bigoplus_{g\in G}\, X_g\ot Y \xrightarrow{\oplus\,c_{X_g,Y}}
\bigoplus_{g\in G}\, T_g(Y) \ot X_g \xrightarrow{\oplus\, v_g \ot \id_{X_g}}
\bigoplus_{g\in G}\, Y \ot X_g = Y\ot X,
\end{equation}
It follows from condition (a) of Definition~\ref{crossed} that
$\tilde c_{X,Y}$ respects the equivariant structures, i.e., it is an
isomorphism in $\C^G$. Its naturality is clear.  The fact that
$\tilde c$ is a braiding on $\C^G$ (i.e., the hexagon axioms)
follows from the commutativity of diagrams \eqref{G-hex1} and
\eqref{G-hex2}. It is easy to check that $\tilde c$ restricts to the
standard braiding on $\Rep(G)=\Vec^G\subset \C^G$. Hence, $\C^G$
contains a Tannakian subcategory $\Rep(G)$.

Conversely, let $\C$ be a braided fusion category with braiding $c$
containing a Tannakian subcategory $\Rep(G)$.   The restriction of
the de-equivariantization functor $F$ from \eqref{de-eq functor} on
$\Rep(G)$ is isomorphic to the fiber functor $\Rep(G)\to \Vec$.
Hence for any object $X$ in $\C_G$ and any object $V$ in $\Rep(G)$
we have an automorphism  of $F(V)\ot X$ defined as the composition
\begin{equation}
\label{FVX}
F(V) \ot X \xrightarrow{\sim} X \ot F(V) \xrightarrow{\sim} F(V) \ot X  ,
\end{equation}
where the first isomorphism comes from the fact that $F(V) \in \Vec$
and the second one is \eqref{Central Func}.

When $X$ is simple we have an isomorphism $\text{Aut}_\C(F(V) \ot X)
\cong \text{Aut}_{\Vec} (F(V))$, hence we  obtain a tensor
automorphism $i_X$  of $F|_{\Rep(G)}$. Since
$\Aut_{\otimes}(F|_{\Rep(G)})\cong G$ we have an assignment
$X\mapsto i_X \in G$. The hexagon axiom of braiding implies that
this assignment is multiplicative, i.e., that $i_Z =i_X i_Y$ for any
simple object $Z$ contained in $X\ot Y$. Thus, it defines a
$G$-grading on $\C$:
\begin{equation}
\label{grading via i}
\C =\bigoplus_{g\in G}\, \C_g,\mbox{ where } \O(\C_g) =\{  X \in \O(\C) \mid i_X =g \}.
\end{equation}
It is straightforward to check that $i_{T_g(X)} =ghg^{-1}$ whenever $i_X=h$.

Finally, to construct a $G$-crossed braiding on $\C$ observe that
$\C$ and $\C^\text{rev}$  are embedded into the crossed product
category $\C\rtimes G = (\C^G)^*_\C$ as subcategories
$\C_\text{left}$ and $\C_\text{right}$ consisting, respectively,  of
functors of left and right multiplications by objects of $\C$.
Clearly, there  is a natural family of isomorphisms
\begin{equation}
\label{rel braiding}
X\ot Y \xrightarrow {\sim}Y \ot X,\qquad X \in \C_\text{left},\, Y\in \C_\text{right},
\end{equation}
satisfying obvious compatibility conditions.
Note that $\C_\text{left}$  is
identified with the diagonal subcategory of $\C\rtimes G$
spanned by objects $X\boxtimes g,\, X\in \C_g,\, g\in G,$ and $\C_\text{right}$
is identified with the trivial component subcategory $\C \boxtimes e$.
Using  \eqref{crossed prod} we conclude that isomorphisms \eqref{rel braiding}
give rise to a $G$-crossed braiding on $\C$.

One can check that the two above constructions (from braided fusion categories
containing $\Rep(G)$ to braided $G$-crossed categories and vice versa) are inverses
of each other, see \cite{Ki, Mu4, DGNO} for details.

\begin{remark}
\label{when CG is nd}
Let $\C=\oplus_{g\in G}\,\C_g$ be a braided $G$-crossed  fusion category.
It was shown in \cite{DGNO} that
the braided category $\C^G$ is non-degenerate if and only if  $\C_e$
is non-degenerate and the $G$-grading of $\C$ is faithful.
\end{remark}

\section{The center of a graded fusion category}
\label{main section}


Let $G$ be a finite group and let $\D$ be a fusion category.
Throughout this section  $\C$ will denote
a fusion category with a faithful $G$-grading, whose trivial
component is $\D$, i.e., $\C$ is a $G$-extension of $\D$:
\begin{equation}
\label{graded C}
\C =\bigoplus_{g\in G}\,\C_g,\qquad \C_e = \D.
\end{equation}
In what follows we consider only {\em faithful}  gradings, i.e., such that
$\C_g\neq 0$, for all $g\in G$. An object of $\C$ contained in
$\C_g$ will be called {\em homogeneous} of degree $g$.

Our goal is to describe the center $\Z(\C)$  as an equivariantization
of the relative center $\Z_\D(\C)$ defined in Section~\ref{relative center}.

\subsection{The relative center $\Z_\D(\C)$ as a braided $G$-crossed category}
\label{rel center triv comp}

Let us define a canonical braided $G$-crossed  category
structure on $\Z_\D(\C)$.

First of all, there is an obvious faithful $G$-grading on $\Z_\D(\C)$:
\begin{equation}
\label{ZDC grading}
\Z_\D(\C) =\bigoplus_{g\in G}\, \Z_\D(\C_g).
\end{equation}
Indeed, it is clear that for every simple object $X$ of $\Z_\D(\C)$ the forgetful
image of $X$ in $\C$ must be homogeneous.

Next, let  us define the action of $G$ on $\Z_\D(\C)$.
Take $g, h\in G$.

Let $\Fun_{\D \bt \D^\rev } (\C_g,\, \C_h)$ denote the category of
$\D$-bimodule  functors from $\C_g$ to $\C_h$. Clearly, it is
a $\Z(\D)$-bimodule category.

\begin{proposition}
\label{towards G action}
Let $g,h\in G$.  The functors
\begin{eqnarray}
\label{Lgh}
L_{g,h}  &:& \Z_\D(\C_h)  \xrightarrow{\sim}  \Fun_{\D \bt \D^\rev }(\C_g,\, \C_{hg}) : Z \mapsto Z\ot ?, \\
\label{Rgh}
R_{g,h}  &:& \Z_\D(\C_h)  \xrightarrow{\sim}  \Fun_{\D \bt \D^\rev }(\C_g,\, \C_{gh}) : Z \mapsto ? \ot Z.
\end{eqnarray}
are equivalences of $\Z(\D)$-bimodule categories.

\end{proposition}
\begin{proof}
We prove that \eqref{Lgh} is an equivalence.   Let  $\Fun_{\D}(\C_g,\, \C_{hg})$  be the category
of  right  $\D$-module functors  from $\C_g$ to  $\C_{hg}$.  It suffices to prove that
\begin{equation}
\label{Mgh}
M_{g,h}: \C_h \to \Fun_{\D}(\C_g,\, \C_{hg}) : X \mapsto  X\ot ?
\end{equation}
is an equivalence. Indeed, $\D$-bimodule functor structures on $M_{g,h}(X)$ for
$X\in~\C_h$ are in bijection with central structures on $X$.

For every  $g\in G$ choose a simple object $X_g \in \C_g$.
Then $A_g := X_g \ot X_g^*$ is an algebra in $\D$.  The category of left $A_g$-modules
in $\C$  is equivalent to $\C$  as a right $\C$-module category and the category
of $A_g$-modules in $\D$  is equivalent to $\C_g$  as a right $\D$-module category.

It follows that for all $g,h\in G$ there is an equivalence  $Y \mapsto X_g \ot Y \ot X_{hg}^*$
between $\C$ and the category of $A_g-A_{hg}$
bimodules in $\C$.

It restricts to an equivalence between  $\C_h$
and the category of $A_g-A_{hg}$ bimodules in  $\D$. It is easy to see that
the latter equivalence  coincides with \eqref{Mgh}.

The proof of equivalence \eqref{Rgh} is completely similar.
 \end{proof}

Let us define tensor functors
\begin{equation}
\label{pre-action on ZDC}
T_{g,h} :=  L^{-1}_{g, ghg^{-1}} R_{g, h} : \Z_\D(\C_h) \to \Z_\D(\C_{ghg^{-1}}), \quad g,h\in G,
\end{equation}
and set
\begin{equation}
\label{action on ZDC}
T_g: =\bigoplus_{h\in G}\, T_{g,h}  : \Z_\D(\C) \to \Z_\D(\C).
\end{equation}

It follows
that there is a natural family of isomorphisms:
\begin{equation}
\label{braiding on ZDC}
c_{X,Y} : X\ot Y \xrightarrow{\sim} T_{g}(Y)\ot X, \qquad  X\in \C_g, \, Y\in \Z_\D(\C),
\quad g\in G,
\end{equation}
satisfying natural compatibility conditions. 
Since the grading \eqref{ZDC grading} is faithful we have $T_g(\Z_\D(\C_h)) \subset \Z_\D(\C_{ghg^{-1}})$.

Take $X_1\in \C_{g_1},\, X_2\in \C_{g_2}$ and set $X =X_1\ot X_2$ in \eqref{braiding on ZDC}.
We obtain a natural isomorphism
\[
T_{g_1}T_{g_2}(Y)\ot X_1 \ot X_2  \xrightarrow{\sim} T_{g_1g_2}(Y) \ot X_1 \ot X_2.
\]
and, hence, an isomorphism of functors $T_{g_1} T_{g_2} \xrightarrow{\sim} T_{g_1 g_2}$.
Thus, the assignment $g \mapsto T_g$
is an action of $G$ on $\Z_\D(\C)$ by tensor autoequivalences.

Suppose that $X$ is an object
in $\Z(\C_g)$. Then both sides of \eqref{braiding on ZDC} have structure of  objects in
$\Z_\D(\C)$ obtained by composing central structures of $X$ and $Y$.

\begin{lemma}
\label{G crossed on ZDC}
Isomorphisms  \eqref{braiding on ZDC} define a $G$-braiding on $\Z_\D(\C)$.
\end{lemma}
\begin{proof}
That isomorphisms  \eqref{braiding on ZDC} are indeed morphisms in $\Z_\D(\C)$
follows from commutativity of the diagram
\begin{equation}
\xymatrix{
X \ot Y \ot V \ar[rr]^{\id_X\ot \delta_V} \ar[d]_{c_{X,Y}\ot \id_V}  &&
X\ot V \ot Y \ar[rr]^{\gamma_V\ot\id_Y } \ar[dll]_{c_{X\ot V, Y}}  &&
V\ot X \ot Y \ar[dll]_{c_{V\ot X, Y}}  \ar[d]^{\id_V \ot c_{X,Y}} \\
T_g(Y)\ot X\ot V \ar[rr]_{\id_{T_g(Y)}\ot \gamma_V} &&
T_g(Y)\ot V\ot X \ar[rr]_{T_g(\delta)_V\ot \id_X}  &&
V \ot T_g(Y)\ot X,
}
\end{equation}
where $(X,\, \gamma) \in \Z_\D(\C_g),\, (Y,\, \delta)\in \Z_\D(\C),$ and $V\in \D$.
Indeed, the parallelogram in the middle commutes by naturality of $c$, and
the two triangles commute since the natural isomorphisms $? \ot Y \xrightarrow{\sim}
T_g(Y)\ot \,? : \C_g \to \C_{gh},\, g,h\in G,$  commute with left and right actions of $\D$.

It is straightforward to check that isomorphisms $c_{X,Y}$ satisfy the compatibility conditions
of Definition~\ref{crossed}.
\end{proof}

The above constructions and arguments prove the following

\begin{theorem}
\label{ZDC theorem}
Let $G$ be a finite group and let $\C$ be a fusion category with a faithful $G$-grading
whose trivial component is $\D$. The relative  center $\Z_\D(\C)$ has a canonical structure
of a braided  $G$-crossed category.
\end{theorem}

\begin{remark}
In particular, to every $G$-extension of a fusion category $\D$ we assigned an
action of $G$ by braided autoequivalences of $\Z(\D)$. This assignment is studied
in detail in \cite{ENO3}.
\end{remark}

\subsection{The center $\Z(\C)$ as an equivariantization}
\label{canonical G-cross}

As before, let $G$ be a finite group and let  $\C$ be a fusion category
with a faithful $G$-grading  \eqref{graded C}. Let $\Z_\D(\C)$ be the
braided $G$-crossed category constructed in Section~\ref{rel center triv comp}.

\begin{theorem}
\label{ZCG =ZDC}
There is an equivalence of braided fusion categories
\begin{equation}
\label{ZDCG=ZC}
\Z_\D(\C)^G \xrightarrow{\sim} \Z(\C).
\end{equation}
\end{theorem}
\begin{proof}
We see from \eqref{braiding on ZDC} that a $G$-equivariant object in $\Z_\D(\C)$
has a structure of a central  object in $\C$ defined as in \eqref{braiding from equiv}.
It follows from definitions that the corresponding tensor functor $\Z_\D(\C)^G \to \Z(\C)$
is braided.

Conversely, given an object $Y$ in $\Z(\C)$
consider its forgetful image $\tilde{Y}$ in $\Z_\D(\C)$.  Combining the central structure
of $Y$ with isomorphism \eqref{braiding on ZDC}  we obtain natural isomorphisms
\[
\tilde{Y} \ot X \xrightarrow{\sim} T_g(\tilde{Y} ) \ot X,\qquad  X\in \C_g,\, g\in G,
\]
which give rise to a $G$-equivariant structure on $\tilde{Y}$.  Hence, we have
a tensor functor $\Z(\C) \to \Z_\D(\C)^G$.  It is clear that
the above two functors are quasi-inverses of each other.
\end{proof}

Let us describe the Tannakian subcategory $\E \cong \Rep(G) \subset \Z(\C)$
corresponding to equivalence \eqref{ZDCG=ZC}.
For any representation $\pi : G \to GL(V)$ of the grading
group $G$ consider an object $I_\pi$ in $\Z(\C)$
where $I_\pi = V \ot \be$ as an object of $\C$
with the permutation isomorphism
\begin{equation}
\label{defining E}
c_{I_\pi, X}:= \pi(g) \ot \id_X:
 I_\pi \ot X \cong X \ot I_\pi,\qquad \mbox{ when }
X\in \C_g.
\end{equation}
Then $\E$ is the subcategory of $\Z(\C)$ consisting
of  objects $I_\pi$, where $\pi$ runs through all
finite-dimensional representations of $G$.

\begin{remark}
Here is another description of the subcategory $\E$: it consists
of all objects in $\Z(\C)$ sent to $\Vec$ by the forgetful functor $\Z(\C)\to \Z_\D(\C)$.
\end{remark}

\begin{corollary}
\label{E 'by E}
Let $\C$ be a faithfully $G$-graded fusion category with the trivial component $\D$.
Let $\E =\Rep(G)\subset \Z(\C)$ be the Tannakian subcategory constructed above.
Then the de-equivariantization category $(\E')_G$ is braided tensor equivalent to $\Z(\D)$.
\end{corollary}
\begin{proof}
The statement follows from Theorem~\ref{ZCG =ZDC} since
$(\E')_G$  is the trivial component of the grading of $\Z(\C)_G =\Z_\D(\C)$.
\end{proof}

\begin{remark}
The above assignment
\begin{equation}
\left\{
 G\text{-extensions  of  } \D
\right\} \mapsto
\left\{
 \text{braided }  G\text{-crossed extensions of  } \Z(\D)
\right\}
\end{equation}
can be thought of as an analogue of the  center construction for $G$-extensions.
\end{remark}

Next, we describe simple objects of $\Z(\C)$. For any conjugacy class $K$ in $G$ fix
a representative $a_K\in K$. Let $G_K$ denote the centralizer of $a_K$ in $G$.
Note that the action \eqref{action on ZDC} of $G$ on $\Z_\D(\C)$
restricts  to the action of  $G_K$  on $\Z_\D(\C_{a_K})$.

\begin{proposition}
\label{freaking simple objects}
There is a bijection between the set of isomorphism classes of simple objects of $\Z(\C)$
and pairs $(K,X)$, where $K$ is a conjugacy class of $G$ and $X$ is a simple
$G_K$-equivariant object of $\Z_\D(\C_{a_K})$.
\end{proposition}
\begin{proof}
By Theorem~\ref{ZCG =ZDC} we
have $\Z(\C) \simeq \Z_\D(\C)^G$  so the stated parameterization
is immediate from the description of simple objects of  the equivariantization category  given in
Proposition~\ref{simples in CG}.
\end{proof}


\subsection{A criterion for a graded fusion category to be group-theoretical.}
\label{when gt}

We have seen in Corollary~\ref{E 'by E} that $\Z(\C)$ contains a
Tannakian subcategory $\E=\Rep(G)$ such that the
de-equivariantization $(\E')_G$ is braided equivalent to $\Z(\D)$,
where $\D$ is the trivial component of $\C$. Furthermore, by
Remark~\ref{BGC rems}, there is a canonical action of $G$ on
$\Z(\D)$, by braided autoequivalences. By \cite{DGNO}, Tannakian
subcategories of $\Z(\C)$ containing  $\E$ bijectively correspond to
$G$-stable Tannakian subcategories of $(\E')_G \simeq \Z(\D)$.
Combining this observation with Theorem~\ref{gt DGNO}(ii) we obtain
the following criterion.

\begin{corollary}
\label{graded gt criterion} A graded fusion category $\C
=\bigoplus_{g\in G}\,\C_g,\quad \C_e = \D$, is group-theoretical if
and only if  $\Z(\D)$ contains a $G$-stable Lagrangian subcategory.
\end{corollary}

We will use Corollary~\ref{graded gt criterion} in  Section~\ref{when TY is gt}
 to characterize group-theoretical Tambara-Yamagami categories.

We can specialize Corollary~\ref{graded gt criterion}  to equivariantization categories.
Let $G$ be a finite group acting on a fusion category $\C$.  The equivariantization  $\C^G$ is Morita equivalent to
the  crossed product category $\C \rtimes G$, see Section~\ref{group actions}, therefore,
$\Z(\C^G)\cong \Z(\C\rtimes G)$.  Clearly, the trivial component of $\Z(\C\rtimes G)_G$
is $\Z(\C)$ and  the canonical action of $G$ on $\Z(\C)$  is induced from the action of
$G$ on $\C$ in an obvious way.

\begin{corollary}
\label{equiv gt criterion}
The equivariantization $\C^G$ is group-theoretical if and only if there exists a $G$-stable Lagrangian
subcategory of $\Z(\C)$.
\end{corollary}

\begin{remark}
Let $G$ act on $\C$ as before. One can check (independently from the results of this section)
that the $G$-set of Lagrangian subcategories of $\Z(\C)$ is isomorphic to the $G$-set of
indecomposable pointed $\C$-module categories.  This isomorphism is given by the map
constructed in \cite[Theorem 4.17]{NN}.  Thus,  the criterion in Corollary~\ref{equiv gt criterion}
is the same as \cite[Corollary 3.6]{Ni}.
\end{remark}

\subsection{Example: the relative center of a crossed product category}

Let $G$ be a finite group and let $g \mapsto T_g,\, g\in G,$ be an
action of $G$ on a fusion category $\D$. Let $\C:= \D \rtimes G$ be
the crossed product category defined in Section~\ref{group actions}.
It has a natural grading
\[
\C =\bigoplus_{g\in G}\, \C_g,\quad \mbox{where } \C_g = \{ Y \bt g
\mid Y\in \D\}.
\]

Let us describe the braided $G$-crossed fusion category structure on
the relative center
\[
\Z_\D(\C) =\bigoplus_{g\in G}\, \Z_\D(\C_g).
\]
By definition, the objects of $\Z_\D(\C_g)$ are pairs $(Y\bt
g,\,\gamma)$, where $Y\in \D$ and
\begin{equation}
\label{gamma crpr} \gamma =\{ \gamma_X : X\ot Y \xrightarrow{\sim} Y
\ot T_g(X) \}_{X \in \D}
\end{equation}
is a natural family of isomorphisms satisfying natural
compatibility conditions. Thus, $\Z_\D(\C_g)$ can be viewed as a
``deformation" of $\Z(\D)$ by means of  $T_g$.

The action of $G$ on $\D$ induces an action $h \mapsto \tilde{T}_h$
on $\Z_\D(\C)$ defined as follows.  Applying $T_h, \, h\in G,$ to
$\gamma_{T_{h^{-1}}(X)}$ in \eqref{gamma crpr} we obtain an
isomorphism
\begin{equation}
\label{action crpr} \tilde{\gamma}_X: X \ot T_h(Y)
\xrightarrow{\sim} T_h(Y)  \ot T_{hgh^{-1}}(X).
\end{equation}
Set $\tilde{T}_h(Y\bt g,\,\gamma) := (T_{h}(Y)\bt
hgh^{-1},\,\tilde\gamma)$. Thus, $\tilde{T}_h$ maps   $\Z_\D(\C_g)$
to $\Z_\D(\C_{hgh^{-1}})$.

Finally, the $G$-braiding between  objects $(X\bt h) \in
\Z_\D(\C_h)$ and $(Y\bt g) \in \Z_\D(\C_g)$  comes from
the following isomorphism
\begin{eqnarray*}
(X \bt h) \ot (Y \bt g)
&=& (X \ot T_h(Y)) \bt hg  \\
&\xrightarrow{\tilde{\gamma}}&  (T_h(Y)  \ot  T_{hgh^{-1}}(X)) \bt hg \\
&=&  (T_h(Y) \bt hgh^{-1}) \ot (X \bt h) \\
&=&  \tilde{T}_h(Y\bt g)  \ot (X \bt h).
\end{eqnarray*}


By Theorem~\ref{ZCG =ZDC}, the category $\Z(\D\rtimes G)\cong \Z(\D^G)$ is
equivalent to the equivariantization of the above braided
$G$-crossed category.

\section{The centers of Tambara-Yamagami categories}
\label{TY section}
Our goal in this section is to apply techniques developed in Section~\ref{main section}
to Tambara-Yamagami categories introduced in \cite{TY} (see Section~\ref{def TY} below
for the definition).
Namely, using the techniques in Section~\ref{main section}
we establish a criterion for a Tambara-Yamagami category to be group-theoretical.
We then use this criterion together with Corollary~\ref{equiv gt criterion}
to produce a series of non  group-theoretical semisimple Hopf algebras.
In this section we assume that our ground field $k$ is the field of
complex numbers $\mathbb{C}$.
We begin by recalling the definition of  Tambara-Yamagami category.

\subsection{Definition of the  Tambara-Yamagami category}
\label{def TY}
In \cite{TY} D.~Tambara and S.~Yamagami completely classified all
$\mathbb{Z}/2\mathbb{Z}$-graded
fusion categories  in which all but one
simple object are invertible. They showed that any such category
$\TY(A,\chi, \tau)$ is determined, up to an equivalence, by a
finite Abelian group $A$, a non-degenerate symmetric bilinear form
$\chi: A\times A \to k^\times$, and a square root $\tau\in k$ of
$|A|^{-1}$. The category $\TY(A,\chi, \tau)$ is described as
follows. It is a skeletal category (i.e., such that any two
isomorphic objects are equal) with simple objects $\{a\mid a\in
A\}$ and $m$, and tensor product
\[
a\ot b =a+b,\quad a\ot m = m,\quad m \ot a=m,\quad m \ot m
=\bigoplus_{a\in A}\, a,
\]
for all $a, b\in A,$ and the unit object $0\in A$.
The associativity constraints are given by
\begin{eqnarray*}[lcl:lcl]
\alpha_{a, b, c} & = & \id_{a+b+c}, &
\alpha_{a, b, m} & = & \id_{m}, \\
\alpha_{a, m, b} & = & \chi(a,b)\,\id_{m}, &
\alpha_{m, a, b} & = & \id_{m}, \\
\alpha_{a, m, m} & = & \bigoplus_{b\in A}\, \id_{b}, &
\alpha_{m, a, m} & = & \bigoplus_{b\in A}\, \chi(a,b)\,\id_{b}, \\
\alpha_{m, m, a} & = & \bigoplus_{b\in A}\, \id_{b}, &
\alpha_{m, m, m} & = & \bigoplus_{a,b\in A}\, \tau\chi(a,b)^{-1}\,\id_{m}.
\end{eqnarray*}
The unit constraints are the identity maps.
The category $\TY(A,\chi,\tau)$ is rigid with $a^*=-a$ and $m^*=m$
(with obvious evaluation and coevaluation maps).

Let $n:=|A|$. The dimensions of simple objects of
$\TY(A,\chi,\tau)$ are $\FPdim(a)=1,\, a\in A$, and
$\FPdim(m)=\sqrt{n}$. We have $\FPdim(\TY(A,\chi,\tau))=2n$.

Let $\bZ/2\bZ =\{ 1,\,\delta\}$. The $\bZ/2\bZ-$grading
on $\TY(A,\chi,\tau)$ is
$$
\TY(A,\chi,\tau) = \TY(A,\chi,\tau)_1 \oplus \TY(A,\chi,\tau)_\delta
$$
where $\TY(A,\chi,\tau)_1$ is the full fusion subcategory
generated by the invertible objects $a \in A$ and $\TY(A,\chi,\tau)_\delta$
is the full abelian subcategory generated by the object $m$.

Let $\C:= \TY(A,\chi,\tau)$ and $\D:= \TY(A,\chi,\tau)_1$.

\subsection{Braided $\bZ/2\bZ$-crossed category $\Z_{\D}(\C)$}
\label{subsection:simples of relative center}

First,  let us describe the simple objects of $\Z_\D(\C)= \Z(\C_1) \oplus \Z_\D(\C_\delta)$.
Let $\widehat{A} := \Hom(A, k^\times)$.  Clearly, $\Z(\C_1) =\Z(\Vec_A)$,  so
its simple objects are parameterized by  $(a,\phi) \in  A \times \widehat{A}$.
The object  $X_{(a,\phi)}$ corresponding to such a pair is equal to $a$ as an object
of $\C$ and its central structure is given by
\begin{equation}
\label{relative central structure on a}
\phi(x)\id_{a+x}  : x\ot X_{(a,\phi)} \xrightarrow{\sim}  X_{(a,\phi)}  \ot x.
\end{equation}

Using Definition~\ref{center of a module} we see that
simple objects of $\Z_\D(\C_\delta)$ are parameterized by functions
$\rho: A \to k^\times$ satisfying
\begin{equation}
\label{rhochi}
\rho(a+b) = \chi(a,b)^{-1} \rho(a) \rho(b),
\qquad a,b \in A
\end{equation}
(clearly, such functions form a torsor over $\widehat{A}$). The corresponding
object $Z_{\rho}$ is equal to $m$ as an object of $\C$ and has the relative
central structure
\begin{equation}
\label{relative central structure on m}
\rho(x)\id_m : x \ot Z_\rho  \xrightarrow{\sim} Z_\rho \ot x,\quad x\in A.
\end{equation}

Let $A\to \widehat{A}: a\mapsto \widehat{a}$ be the homomorphism defined by
$\widehat{a}(x) =\chi(x,\, a)$.   Similarly, let $\widehat{A}\to A: \phi\mapsto \widehat{\phi}$
be the homomorphism defined by $\phi(x)=\chi(x, \, \widehat{\phi})$ (recall that $\chi$
is non-degenerate). Clearly, these two maps are inverses of each other.

The fusion rules of $\Z_\D(\C)$ are computed using formula \eqref{centrproduct} :
\begin{eqnarray*}
X_{(a,\phi)} \ot X_{(b,\psi)}  &=& X_{(a+b,\phi+\psi)},  \\
X_{(a,\phi)} \ot Z_\rho &=&   Z_{\rho\phi (-\widehat{a})},\\
Z_\rho  \ot X_{(a,\phi)} &=&  Z_{\rho\phi (-\widehat{a})},\\
Z_{\rho'}  \ot Z_{\rho} &=& \bigoplus_{a\in A}\, X_{(a,\widehat{a}\rho'/\overline{\rho})}.
\end{eqnarray*}
We have $X_{(a,\phi)}^* = X_{(-a,-\phi)}$ and $Z_\rho^*=Z_{\overline\rho}$, where $\overline\rho(x)=\rho(-x),\, x\in A$.

Using the construction given in Section~\ref{rel center triv comp} we see that the
action of $\bZ/2\bZ$ on $\Z_\D(\C)$ is given by
\begin{equation}
\label{action of Z2}
T_1=\id_{\Z_\D(\C)};\qquad T_\delta(X_{(a,\phi)}) =  X_{(-\widehat{\phi},\, -\widehat{a})},\quad
T_\delta(Z_\rho)=Z_\rho.
\end{equation}
The monoidal functor structure on $\bZ/2\bZ \to \Aut_\otimes(\Z_\D(\C))$
is given by the natural isomorphism $\gamma:=\gamma_{\delta,\delta} : T_\delta\circ T_\delta
\xrightarrow{\sim} T_1$ defined by
\[
\gamma_{X_{(a,\phi)}} =\phi(a)\id_{X_{(a,\phi)}},\qquad \gamma_{Z_\rho} =
\left( \tau \sum_{x\in A}\, \rho(x)^{-1} \right) \id_{Z_\rho}.
\]

The crossed braiding morphisms on $\Z_\D(\C)$ are given by
\begin{eqnarray*}
c_{ X_{(a,\phi)}, X_{(b,\psi)}  }  &=& \psi(a) \id_{a+b} : X_{(a,\phi)}\ot  X_{(b,\psi)}  \xrightarrow{\sim} X_{(b,\psi)}  \ot X_{(a,\phi)} \\
c_{ X_{(a,\phi)},  Z_\rho }          &=&  \rho(a)\id_m :  X_{(a,\phi)}\ot   Z_\rho  \xrightarrow{\sim}  Z_\rho \ot  X_{(a,\phi)} \\
c_{ Z_\rho,    X_{(a,\phi)} }       &=&   \id_m :  Z_\rho \ot  X_{(a,\phi)}  \xrightarrow{\sim}  X_{(-\widehat{\phi},\, -\widehat{a})} \ot Z_\rho \\
c_{ Z_{\rho'},  Z_{\rho} }             &=&  \oplus_{a\in A}\, \rho(-a)^{-1} \id_a :   Z_{\rho'}\ot   Z_{\rho}  \xrightarrow{\sim}   Z_{\rho}\ot   Z_{\rho'} .
\end{eqnarray*}

\subsection{The equivariantization category $\Z_{\D}(\C)^{\bZ/2\bZ}$}
\label{subsection:equivariantization of relative center}

A simple calculation of $\bZ/2\bZ$-equivariant objects in $\Z_\D(\C)$
establishes the following.
\begin{proposition}
\label{simples in ZTY}
The following is a complete list of simple objects of $\Z_\D(\C)^{\bZ/2\bZ}
\cong \Z(\TY(A,\chi,\tau))$
up to an isomorphism:
\begin{enumerate}
\item[(1)] $2n$ invertible objects parameterized by pairs $(a, \epsilon)$, where $a\in A$ and
$\epsilon^2 =\chi(a,a)^{-1}$. The corresponding object $X_{a,\epsilon}$ is equal to
$X_{(a,-\widehat{a})}$ as an object of $\Z_\D(\C)$ and has $\bZ/2\bZ-$equivariant
structure
$$
\epsilon \id_{X_{(a,-\widehat{a})}}:
T_{\delta}(X_{(a,-\widehat{a})}) \xrightarrow{\sim} X_{(a,-\widehat{a})};
$$
\item[(2)] $\frac{n(n-1)}{2}$ two-dimensional objects parameterized by
unordered pairs $(a,\,b)$ of distinct objects in $A$.
The corresponding object $Y_{a,b}$ is equal to
$X_{(a,-\widehat{b})} \oplus X_{(b,-\widehat{a})}$ as an object of
$\Z_\D(\C)$ and has $\bZ/2\bZ-$equivariant
structure
$$
\left( \id_{X_{(a,-\widehat{b})}} \oplus
\chi(a,b)^{-1} \id_{X_{(b,-\widehat{a})}} \right):
T_{\delta}(X_{(a,-\widehat{b})} \oplus X_{(b,-\widehat{a})})
\xrightarrow{\sim} X_{(a,-\widehat{b})} \oplus X_{(b,-\widehat{a})};
$$
\item[(3)] $2n$ $\sqrt{n}-$dimensional objects parameterized by pairs $(\rho, \Delta)$,
where $\rho: A \to k^\times$  satisfies \eqref{rhochi} and
$\Delta^2 = \tau \sum_{x\in A}\, \rho(x)^{-1}$.
The corresponding object $Z_{\rho,\Delta}$ is equal to
$Z_{\rho}$ as an object of $\Z_\D(\C)$ and has $\bZ/2\bZ-$equivariant
structure
$$
\Delta \id_{Z_{\rho}}:
T_{\delta}(Z_{\rho}) \xrightarrow{\sim} Z_{\rho}.
$$
\end{enumerate}
\end{proposition}

Recall from \cite{ENO1} that in a braided fusion category of an integer
Frobenius-Perron dimension there
is a canonical choice of a twist $\theta$ such that the
categorical dimensions of objects coincide with their
Frobenius-Perron dimensions. Namely, for any simple object $X$
the scalar $\theta_X$ is defined in such a way that the
composition
\begin{equation}
\be \xrightarrow{\coev_X} X \ot X^* \xrightarrow{\theta_X c_{X,X^*}}
X^* \ot X  \xrightarrow{\ev_X} \be
\end{equation}
is equal to $\FPdim(X)\id_X$.

Let $\theta$ be the canonical twist on $\Z(\C)$. Using the above observation,
explicit formulas from Subsection~\ref{subsection:simples of relative center}, and
Section~\ref{G-equivariant vector bundles}, we immediately obtain the following.
\begin{eqnarray*}[rcl:rcl:rcl]
\theta_{X_{a,\epsilon}} &=& \chi(a,a)^{-1},
& \theta_{Y_{a,b}} &=& \chi(a,b)^{-1},
& \theta_{Z_{\rho, \Delta}} &=& \Delta.
\end{eqnarray*}
Using the fusion rules of $\Z(\C)$ (which may be computed using the explicit formulas
in Subsection~\ref{subsection:simples of relative center}), values
of the twists above, and the well known formula
\begin{equation}
S_{X,Y}= \theta_X^{-1} \theta_Y^{-1} \sum_Z\, N_{X,Y}^Z \theta_Z d_Z
\end{equation}
we obtain the $S$- and $T$-matrices of $\Z(\C)$:
\begin{eqnarray*}[rcl:rcl]
S_{X_{a,\epsilon}, X_{a',\epsilon'}} &=& \chi(a,a')^2,
& S_{X_{a,\epsilon}, Y_{b,c}} &=& 2 \chi(a,b+c),\\
S_{X_{a,\epsilon}, Z_{\rho,\Delta}} &=& \epsilon \sqrt{n} \rho(a),
& S_{Y_{a,b}, Y_{c,d}} &=& 2\left(\chi(a,d) \chi(b,c) + \chi(a,c) \chi(b,d)\right),\\
S_{Y_{a,b}, Z_{\rho,\Delta}} &=& 0,
& S_{Z_{\rho,\Delta}, Z_{\rho',\Delta'}} &=&
\frac{1}{\Delta \Delta'} \sum_{a \in A} \chi(a,a)^2 \rho(a) \rho'(a).
\end{eqnarray*}
\begin{eqnarray*}[rclqrclqrcl]
T_{X_{a,\epsilon}} &=& \chi(a,a)^{-1},
& T_{Y_{a,b}} &=& \chi(a,b)^{-1},
& T_{Z_{\rho, \Delta}} &=& \Delta.
\end{eqnarray*}

\begin{proposition}
\label{when ptd is nd}
The maximal pointed subcategory of $\Z(\C)$ is non-degenerate if and only if
$|A|$ is odd.
\end{proposition}
\begin{proof}
Let $a\in A$ be an element of order $2$. Then $X_{a,\epsilon}$ centralizes every invertible
object of  $\Z(\C)$.
\end{proof}

\begin{remark}
We note that simple objects  and the $S$- and $T$-matrices of $\Z(\C)$  were described by Izumi in \cite{I} using
very different methods.
\end{remark}

\subsection{A criterion for a Tambara-Yamagami category to be group-theoretical}
\label{when TY is gt}
The group $A\times \widehat{A}$ is equipped with  a canonical non-degenerate quadratic form
$q: A\times \widehat{A} \to k^\times$ given by
\[
q((a,\phi)):= \phi(a),\quad (a,\phi)\in A\times \widehat{A}.
\]
We will call a subgroup $B\subset A\times \widehat{A}$ {\em Lagrangian} if $q|_B =1$
and $B=B^\perp$ with respect to the bilinear form defined by $q$.  Lagrangian
subgroups of $A\times \widehat{A}$ correspond to Lagrangian subcategories of
$\Z(\Vec_A)\cong \Vec_{A\times \widehat{A}}$.

The braided tensor autoequivalence $T_\delta$ of $\Z(\Vec_A)$ defined in Section~\ref{subsection:simples of relative center}
determines an order $2$  automorphism of $A\times \widehat{A}$, which we denote simply by $\delta$:
\begin{equation}
\label{delta}
\delta((a,\phi)) =  (-\widehat{\phi},\, -\widehat{a}),\quad (a,\phi)\in A\times \widehat{A}.
\end{equation}

\begin{definition}
\label{lagr in A}
We will say that a subgroup $L\subset A$ is {\em Lagrangian (with respect to $\chi$)}
if $L=L^\perp$ with respect to the inner product on $A$ given by $\chi$.  Equivalently,
$|L|^2 = |A|$ and $\chi|_{L} = 1$.
\end{definition}

\begin{lemma}
\label{power of 2}
Let $A$ be an Abelian $2$-group such that $|A|=2^{2n}$ and let $\chi$
be a non-degenerate symmetric bilinear form on $A$. Then $A$ contains a Lagrangian subgroup.
\end{lemma}
\begin{proof}
It suffices to show that $A$ contains an isotropic element, i.e., an element $x\in A,\,x\neq 0,$
such that $\chi(x,x)=1$. Then one can pass from $A$ to $\langle x\rangle^\perp/\langle x\rangle$
and use induction.

Suppose that $A$ is cyclic with a generator $a$. Then $2^{2n}a =0$ and $\chi(a,\, a)$ is
a $2^{2n}$-th root of unity, hence $\chi(2^na,\, 2^na) = \chi(a,a)^{2^{2n}} =1$.

If $A$ is not cyclic then it contains a subgroup $A_0 = \bZ/2\bZ \oplus \bZ/2\bZ$.
Let $x_1,\, x_2$ be distinct non-zero elements of $A_0$. Suppose $\chi(x_i,\, x_i)\neq 1,\, i=1,2$.
Then $\chi(x_i,\, x_i)=-1$ and $\chi(x_1+x_2,\, x_1+x_2)=1$, as desired.
\end{proof}

\begin{theorem}
\label{thm:when TY is gt}
Let $\C =\TY(A,\, \chi,\, \tau)$ be a  Tambara-Yamagami fusion category. Then $\C$ is group-theoretical
if and only if $A$ contains a Lagrangian subgroup (with respect to $\chi$).
\end{theorem}
\begin{proof}
By Corollary~\ref{graded gt criterion}, $\C$ is group-theoretical if
and only if $\Z(\D)$ contains a $T_\delta$-stable Lagrangian
subcategory. Equivalently, $\C$ is group-theoretical if and only if
$A\times \widehat{A}$ contains a Lagrangian subgroup $B$ stable
under the action
\begin{equation}
\label{stability}
(a,\phi)\mapsto  (\widehat{\phi},\, \widehat{a}).
\end{equation}
This condition on $B$ is the same as being stable under the action
of $\delta$ from \eqref{delta}.

Let $L$ be a Lagrangian (with respect to $\chi$) subgroup of $A$ and
let $\widehat{L}:= \{ \widehat{a} \mid a\in L\}$. Then $L\times
\widehat{L}$ is a Lagrangian subgroup of $A\times \widehat{A}$
stable under \eqref{stability}. Hence $\C$ is group-theoretical.

Conversely, suppose that  $\C$ is group-theoretical. Let us write $A =A_\text{even}\oplus
A_\text{odd}$, where $A_\text{even}$ is the Sylow $2$-subgroup of $A$ and $A_\text{odd}$
is the maximal odd order subgroup of $A$.  Since $|A|$ must be a square, we conclude
that $|A_\text{even}|$ is a square, and so  $A_\text{even}$ contains a Lagrangian
subgroup with respect to $\chi|_{A_\text{even}}$ by Lemma~\ref{power of 2}.

So it remains to show that $A_\text{odd}$ contains a Lagrangian
subgroup with respect to $\chi|_{A_\text{odd}}$. For this end we may assume
that $|A|$ is odd. Let $B \subset A\times \widehat{A}$
be a Lagrangian subgroup stable under \eqref{stability}.
Then $B =B_+\oplus B_-$, where
\[
B_\pm := \{ (a,\pm \widehat{a}) \mid (a,\pm\widehat{a})\in B \}.
\]
Let $L_\pm =B_\pm \cap (A\times \{1\})$.
Then $|L_+||L_- |=|A|$, and $\chi|_{L_\pm} =1$. Hence, $L_\pm$ are Lagrangian subgroups of $A$.
\end{proof}

\begin{remark}
It was observed in \cite[Remark 8.48]{ENO1} that for an odd prime
$p$ and elliptic bicharacter $\chi$ on  $A
=(\bZ/p\bZ)^2$ the category  $\TY(
(\bZ/p\bZ)^2,\chi,\tau)$ is not group-theoretical. The
criterion from Theorem~\ref{thm:when TY is gt} extends this
observation.
\end{remark}

\subsection{A series of non  group-theoretical semisimple Hopf algebras
obtained from Tambara-Yamagami categories}

Here we apply Corollary~\ref{equiv gt criterion} to produce a series of
non group-theoretical fusion categories admitting fiber functors
(i.e., representation categories of non group-theoretical
semisimple Hopf algebras), generalizing examples constructed in \cite{Ni}.

Let $A$ be a finite Abelian group with a non-degenerate
bilinear form $\chi$.
Let $\Aut(A, \chi)$ denote the group of automorphisms of $A$ preserving $\chi$.

The following proposition was proved in \cite[Proposition 2.10]{Ni}.

\begin{proposition}
\label{orthogonal group acts}
There is an action of $\Aut(A, \chi)$
on $\TY(A,\chi,\tau)$ given by $g\mapsto T_g$, where
\[
T_g(A) = g(a),\quad T_g(m)=m,\qquad a\in A,\, g\in \Aut(A, \chi),
\]
with the tensor structure of $T_g$ given by identity morphisms.
\end{proposition}

\begin{corollary}
\label{TYG non gt}
Let $G$ be a subgroup of  $\Aut(A, \chi)$.
Then the fusion category $\TY(A,\chi,\tau)^G$ is
group-theoretical if and only if there is a
Lagrangian subgroup of $(A,\, \chi)$
stable under the action of $G$.
\end{corollary}
\begin{proof}
Combine Corollary~\ref{equiv gt criterion} and
Theorem~\ref{thm:when TY is gt}.
\end{proof}

We will say that a non-degenerate symmetric bilinear form $\chi: A \times A\to k^\times$
is {\em hyperbolic}  if there are Lagrangian subgroups
$L,\, L' \subset A$ such that $A=L \oplus L'$.
Note that in this case $L'$ is isomorphic to the group $\widehat{L}=\Hom(L,\,k^\times)$
of characters of $L$ and $\chi$ is identified with the
canonical bilinear form on $L\oplus \widehat{L}$.

It was shown by D.~Tambara in \cite{Ta1}
that when $n =|A|$ is odd the category $\TY(A,\chi,\tau)$ admits
a fiber functor (i.e., $\TY(A,\chi,\tau)$
is equivalent to the representation category of a semisimple Hopf algebra)
if and only if $\tau^{-1}$ is a positive integer and $\chi$ is hyperbolic.

\begin{corollary}
\label{series of Hopf algebras}
Let $p$ be an odd prime, let $L=(\mathbb{Z}/p\mathbb{Z})^N,\, N\geq 1$,
let $A=L\oplus \widehat{L}$, and let $\chi: A \times A \to k^\times $
be the canonical bilinear form defined by
\[
\chi((a,\, \phi),\, (b,\, \psi)) = \psi(a)\phi(b),\quad a,b\in A,\, \phi,\psi\in \widehat{A}.
\]
Suppose that $G$ is a subgroup of
$\Aut(A,\, \chi)$ not contained in any conjugate of
$\Aut(L) \subset \Aut(A,\, \chi)$. Then
the equivariantization category $\TY(A,\,\chi,\, p^{-N})^G$
is a non group-theoretical fusion category equivalent to
the representation category of a semisimple Hopf algebra
of dimension $2p^{2N}|G|$.
\end{corollary}
\begin{proof}
Note that $\Aut(A,\, \chi)$ acts transitively on the set
of Lagrangian subgroups of $(A,\, \chi)$ and the stabilizer
of $L$ is $\Aut(L)$. Apply Corollary~\ref{TYG non gt}.
\end{proof}

\begin{remark}
The series of fusion categories in Corollary~\ref{series of Hopf algebras}
extends the one constructed in \cite{Ni}, where the case of
$N=1$  and $G =\bZ/2\bZ$ was considered.
\end{remark}

\section{Examples of modular categories arising from quadratic forms}
\label{modquadr}

As before, let  $\C:=\TY(A,\chi,\tau)$ be  a Tambara-Yamagami category
and let $\D:= \TY(A,\chi,\tau)_1$  be the trivial component of
$\bZ/2\bZ-$grading of $\TY(A,\chi,\tau)$.
In this section we assume that our ground field $k$ is the field of
complex numbers $\mathbb{C}$.

Suppose that the symmetric bicharacter $\chi: A \times A \to k^\times$ comes from a
quadratic form on $A$, i.e., there is a function $q: A \to k^\times$ such
that
\[
q(a+b) = q(a)q(b) \chi(a, b),\, \quad a,b\in A \quad \mbox{ and }
\quad q(-a) =q(a).
\]
>From the description obtained in Section~\ref{subsection:simples of relative center}
we observe that $\Z_\D(\C)$  contains a fusion subcategory
spanned by the simple objects $X_{(a, \widehat{a})}, a \in A$, and $Z_{q^{-1}}$.
It is clear from the Tambara-Yamagami classification in Section~\ref{def TY}
that this category is equivalent to $\C$.

\begin{proposition}
\label{GcrTY}
Suppose that the symmetric bicharacter $\chi$ comes from a
quadratic form on $A$. Then  $\C$ admits a $\bZ/2\bZ$-crossed
braided category structure. The equivariantization $\C^{\bZ/2\bZ}$ is non-degenerate
if and only if $|A|$ is odd.
\end{proposition}
\begin{proof}
Clearly, $\C$ inherits the
$\bZ/2\bZ$-crossed braided category structure from $\Z_\D(\C)$.
The non-degeneracy claim follows from Proposition~\ref{when ptd is nd} and Remark~\ref{when CG is nd}.
\end{proof}

Let us assume that $n:=|A|$ is odd. Then $\chi$ corresponds to
a unique quadratic form $q$.  Let  $\E(q,\pm):= \C^{\bZ/2\bZ}$
be the modular category constructed
in Proposition~\ref{GcrTY} (the $\pm$
corresponding to $\tau = \pm\frac{1}{\sqrt{n}}$, respectively).
In what follows we describe the fusion rules and $S$- and $T$-matrices  of $\E(q,\pm)$.

\subsection{Fusion rules of $\E$}
\label{fusion rules of E}
Clearly,  $\E(q,\pm)$ is a fusion category of dimension $4n$.  It has the following
simple objects:
\begin{enumerate}
\item[] two invertible objects, $\mathbf{1} =X_+$ and $X_-$,
\item[] $\frac{n-1}{2}$ two-dimensional objects
$Y_a, \, a\in A-\{0\}$ (with $Y_{-a} =Y_a$)
\item[] two $\sqrt{n}$-dimensional objects $Z_l$, $l \in \bZ/2\bZ$.
\end{enumerate}
Here we simplify the notation used in Subsection~\ref{subsection:equivariantization of relative center}
and denote
\[
X_\pm:= X_{0, \pm 1},\,  Y_a :=  Y_{a, -a},\, \mbox{and }
Z_l := Z_{q^{-1}, \Delta_l},
\]
where $\Delta_l, l \in \bZ/2\bZ$, are distinct square roots of
$\pm \frac{1}{\sqrt{n}} \sum_{a\in A}\,q(a)$.

The fusion rules of $\E(q,\pm)$ are given by:
\begin{eqnarray*}[lcl:lcl:lcl]
X_- \ot X_- &=& X_+, & X_\pm \ot Y_a &=&  Y_a,
& X_+ \ot Z_l &=& Z_l,  \\
X_- \ot Z_l &=& Z_{l+1},
& Y_a \ot Y_b &=& Y_{a+b} \oplus Y_{a-b},
& Y_a \ot Y_a &=& X_+\oplus X_- \oplus Y_{2a},\\
Y_a\ot Z_l &=& Z_0 \oplus Z_1,
& Z_l \ot Z_l  &=& X_+ \oplus \left( \oplus\, Y_a\right),
& Z_l \ot Z_{l+1}  &=& X_- \oplus \left(\oplus\, Y_a \right),
\end{eqnarray*}
where $a,\, b\in A \; (a\neq b)$ and $l\in \bZ/2\bZ$.
All objects of $\E(q,\pm)$ are self-dual.

\begin{remark}
Note that the fusion rules of $\E(q,\pm)$ do not depend on the quadratic form $q$
and the number $\tau$. We show below that the $S$- and $T$-matrices of $\E(q,\pm)$
do depend on  $q$ and $\tau$.
\end{remark}

\subsection{$S$- and $T$-matrices of $\E$}

\begin{lemma}
\label{Gauss sums of q and q^2}
The Gauss sums corresponding to $q$ and $q^2$ are equal up to a sign, i.e.,
$$
\frac{\sum_{a\in A}\, q(a)^2}{\sum_{a\in A}\, q(a)} \in \{\pm 1\}.
$$
\end{lemma}
\begin{proof}
Consider the group $A\times A$ with a non-degenerate
quadratic form $Q= q \times q$. The Gaussian sum for this
form is
\[
\tau(A\times A, Q) = \sum_{a, b\in A}\,q(a)q(b) = \tau(A, q)^2.
\]
The restriction
of $Q$ on the diagonal subgroup $D := \{ (a,a)\mid a\in A\}$
is non-degenerate since $|A|$ is odd. The restriction of $Q$
on the orthogonal complement $D^\perp = \{ (a,-a)\mid a\in A\}$
is non-degenerate as well. By the multiplicativity of Gaussian
sums we have
\[
\tau(A\times A, Q) = \tau(D, Q) \tau(D^\perp, Q) =
(\sum_{a\in A}\,q(a)^2)^2,
\]
which implies the result.
\end{proof}

Using the formulas for the $S$- and $T$- matrices of $\Z(\C)$ given
in Subsection~\ref{subsection:equivariantization of relative center} we
can write down the $S$- and $T$- matrices of $\E(q,\pm)$:
\begin{eqnarray*}[lcl:lcl:lcl]
S_{X_\pm, X_\pm} &=& 1,  & S_{X_\mp, X_\pm} &=& 1,
& S_{X_\pm, Y_a} &=& 2,\\
S_{X_+, Z_l} &=& \sqrt{n},
& S_{X_-, Z_l} &=& -\sqrt{n},
& S_{Y_a, Y_b} &=& 2\left(\frac{q(a+b)^2}{q(a)^2q(b)^2}
+ \frac{q(a)^2q(b)^2}{q(a+b)^2}\right),
\end{eqnarray*}
\begin{eqnarray*}[rcl:rcl]
S_{Y_a, Z_l} &=& 0,
& S_{Z_l, Z_l} &=&
\begin{cases}
\pm \sqrt{n}, \text{ if the Gauss sums of $q$ and $q^2$ coincide,}\\
\mp \sqrt{n}, \text{ otherwise,}
\end{cases}\\
& & & S_{Z_l, Z_{l+1}} &=&
\begin{cases}
\mp \sqrt{n}, \text{ if the Gauss sums of $q$ and $q^2$ coincide,}\\
\pm \sqrt{n}, \text{ otherwise.}
\end{cases}
\end{eqnarray*}
\begin{eqnarray*}[lclqlclqlcl]
T_{X_\pm} &=& 1, & T_{Y_a} &=& q(a)^2, & T_{Z_l} &=& \Delta_l.
\end{eqnarray*}
(Recall that $\Delta_l, l \in \bZ/2\bZ$, are distinct
square roots of $\pm \frac{1}{\sqrt{n}} \sum_{a\in A}\,q(a)$.)

\subsection{Example with $A = \bZ/p\bZ
\times \bZ/p\bZ$}

Let $p$ be an odd prime and let $A:=\bZ/p\bZ
\times \bZ/p\bZ$.
Let ${\legendre{\cdot}{p}}$ denote the Legendre
symbol modulo $p$, i.e., ${\legendre{a}{p}} = 1$
if $a \in (\bZ/p\bZ)^\times$ is a square modulo $p$ and $-1$
otherwise.

Let $a,b \in (\bZ/p\bZ)^\times$ and $\xi := e^{\frac{2\pi i}{p}}$. Consider the following
nondegenerate quadratic form $q$ on $A$:
$$
q(x_1,x_2)= \xi^{ax_1^2-bx_2^2}.
$$
It is hyperbolic if ${\legendre{ab}{p}} = 1$ and elliptic
if  ${\legendre{ab}{p}} = -1$.

We will need the following.

\begin{lemma}
\label{classical gauss sums}
For every $a,b \in A^\times$, we have
$$
\sum_{x \in \bZ/p\bZ} \xi^{ax^2} =
\begin{cases}
{\legendre{a}{p}} \sqrt{p}, \text{ if } p \equiv 1 \, (\mod  4),\\
{\legendre{a}{p}} i \sqrt{p}, \text{ if } p \equiv 3 \, (\mod 4)
\end{cases}
$$
and
$$
\sum_{(x_1,x_2) \in \bZ/p\bZ \times \bZ/p\bZ}
\xi^{ax_1^2-bx_2^2} = {\legendre{ab}{p}} p.
$$
\end{lemma}
\begin{proof}
The first assertion is well known, see for example \cite{R}.
The second assertion is an easy consequence of the first.
\end{proof}

Using Lemma~\ref{classical gauss sums} we can explicitly write the $S$-matrix of
$\E(q,\pm)$:
\begin{eqnarray*}[rcl:rcl:rcl]
S_{X_\pm, X_\pm} &=& 1, & S_{X_\mp, X_\pm} &=& 1,
& S_{X_\pm, Y_{(x_1,x_2)}} &=& 2,\\
S_{X_+, Z_l} &=& p, & S_{X_-, Z_l} &=& -p,
& S_{Y_{(x_1,x_2)}, Y_{(y_1,y_2)}} &=& 4 \Real(\xi^{4ax_1y_1-4bx_2y_2}),\\
S_{Y_{(x_1,x_2)}, Z_l} &=& 0,
& S_{Z_l, Z_l} &=& \pm p,
& S_{Z_l, Z_{l+1}} &=& \mp p,
\end{eqnarray*}
and its $T$-matrix:
\[
T_{X_\pm} = 1,\quad T_{Y_(x_1,x_2)} = \xi^{2ax_1^2-2bx_2^2},\quad
T_{Z_l} = \Delta_l,
\]
%
%
where $\Delta_l, l \in \bZ/2\bZ$, are distinct
square roots of $\pm {\legendre{ab}{p}}$.

The central charge of the modular category $\E(q,\pm)$ is
\[
\zeta(\E(q,\pm)) = {\legendre{ab}{p}}.
\]

Below we give the $S$- and $T$-matrices of the modular category $\E(q, \pm)$
for $p=3$. Order simple objects of $\E(q,\pm)$ as follows:
$\mathbf{1}, X_-, Y_{(0,1)}, Y_{(1,0)}, Y_{(1,1)}, Y_{(1,2)}, Z_+, Z_-$.
There are four modular categories $\E(q,\pm)$ of dimension $36$
corresponding to the choices of hyperbolic/elliptic $q$ and $\tau =\pm \frac{1}{3}$.
\begin{enumerate}
\item[(a)]
When $q$ is hyperbolic we have:
\begin{eqnarray*}
S &=& \begin{pmatrix}
1 & 1 & 2  & 2  & 2 & 2 &  3 & 3 \\
1 & 1 & 2  & 2  & 2 & 2 & -3 & 3 \\
2 & 2 & -2 & 4 & -2 & -2 & 0 & 0 \\
2 & 2 & 4  & -2  & -2 & -2 & 0 & 0 \\
2 & 2 & -2 &-2 & 4 & -2 & 0 & 0 \\
2 & 2 & -2 &-2 & -2 & 4 &0 & 0 \\
3 & -3 & 0& 0& 0 & 0 & \pm 3 & \mp 3 \\
3 & -3 & 0& 0& 0 & 0 & \mp 3 & \pm 3,
\end{pmatrix} \\
T &=& \mbox{diag}\{ 1, 1, \xi^2, \xi, 1 ,1 , 1, -1 \}
\quad \mbox{ when } \tau =\frac{1}{3}, \\
T &=&  \mbox{diag}\{ 1, 1, \xi^2, \xi, 1 ,1 , i, -i \}
\quad \mbox{ when } \tau =-\frac{1}{3}.
\end{eqnarray*}
Note that both the corresponding modular categories are group-theoretical with central charge $1$;
in fact the one with $\tau = \frac{1}{3}$ is equivalent
to the representation category of the double $D(S_3)$
of the symmetric group $S_3$ and the one with
$\tau = -\frac{1}{3}$ is equivalent to the twisted double
of $S_3$.\\
\item[(b)]
When $q$ is elliptic we have:
\begin{eqnarray*}
S &=& \begin{pmatrix}
1 & 1 & 2  & 2  & 2 & 2 &  3 & 3 \\
1 & 1 & 2  & 2  & 2 & 2 & -3 & 3 \\
2 & 2 & -2 & 4 & -2 & -2 & 0 & 0 \\
2 & 2 & 4  & -2  & -2 & -2 & 0 & 0 \\
2 & 2 & -2 &-2 & -2 & 4 & 0 & 0 \\
2 & 2 & -2 &-2 & 4 & -2 &0 & 0 \\
3 & -3 & 0& 0& 0 & 0 & \pm 3 & \mp 3 \\
3 & -3 & 0& 0& 0 & 0 & \mp 3 & \pm 3,
\end{pmatrix} \\
T &=& \mbox{diag}\{ 1, 1, \xi, \xi, \xi^2 ,\xi^2 , i, -i \}
\quad \mbox{ when } \tau =\frac{1}{3}, \\
T &=&  \mbox{diag}\{ 1, 1, \xi, \xi, \xi^2 ,\xi^2 , 1, -1 \}
\quad \mbox{ when } \tau =-\frac{1}{3}.
\end{eqnarray*}
Both the corresponding modular categories are not group-theoretical. They both have
central charge $-1$ and so are not equivalent to
centers of fusion categories. In particular,
they are not equivalent to representation
categories of any twisted group doubles.
\end{enumerate}
\subsection{Example with $A=\bZ/p\bZ$}
Let $p$ be an odd prime and let $A:=\bZ/p\bZ$.
Let $a \in (\bZ/p\bZ)^\times$ and $\xi := e^{\frac{2\pi i}{p}}$. Upto isomorphism
the are two nondegenerate quadratic forms $q$ on $A$:
$$
q(x)= \xi^{ax^2},
$$
one corresponding to ${\legendre{a}{p}} = 1$ and
another to ${\legendre{a}{p}} = -1$.

Using Lemma~\ref{classical gauss sums} we can explicitly write the $S$-matrix of
$\E(q,\pm)$:
\begin{eqnarray*}[lcl:lcl:lcl]
S_{X_\pm, X_\pm} &=& 1, & S_{X_\mp, X_\pm} &=& 1,
& S_{X_\pm, Y_x} &=& 2,\\
S_{X_+, Z_l} &=& \sqrt{p}, & S_{X_-, Z_l} &=& -\sqrt{p},
& S_{Y_x, Y_y} &=& 4 \Real(\xi^{4axy}),\\
S_{Y_a, Z_l} &=& 0,
& S_{Z_l, Z_l} &=& \pm {\legendre{2}{p}} \sqrt{p},
& S_{Z_l, Z_{l+1}} &=& \mp {\legendre{2}{p}} \sqrt{p}.
\end{eqnarray*}
\begin{eqnarray*}[lclqlclqlcl]
T_{X_\pm} &=& 1, & T_{Y_x} &=& \xi^{-2ax^2},
& T_{Z_l} &=& \Delta_l,
\end{eqnarray*}
where
$$
\Delta_l, l \in \bZ/2\bZ, \text{ are distinct}
\begin{cases}
\text{square roots of } \pm {\legendre{a}{p}}, \text{ if } p \equiv 1 \, (\mod  4),\\
\text{square roots of } \pm {\legendre{a}{p}} i, \text{ if }  p \equiv 3 \, (\mod  4).
\end{cases}
$$

The central charge of the modular category $\E(q,\pm)$ is
$$
\zeta(\E(q,\pm)) =
\begin{cases}
{\legendre{2a}{p}}, \text{ if } p \equiv 1 \, (\mod  4),\\
-{\legendre{2a}{p}} i, \text{ if } p \equiv 3 \, (\mod 4).
\end{cases}
$$

Below we give the $S$- and $T$-matrices of the modular category $\E(q, \pm)$
for $p=3 \text{ and } 5$. For $p=3$ we order the simple objects as
$\be, X_-, Y_1, Z_0, Z_1$ and for $p=5$ we order them as
$\be, X_-, Y_1, Y_2, Z_0, Z_1$. (In (c) and (d) below, $\xi = e^{\frac{2\pi i}{5}}$.)

\begin{enumerate}
\item[(a)] When $p=3$ and $a=1$ we have:
\begin{eqnarray*}
S &=&
\begin{pmatrix}
1 & 1 & 2  & \sqrt{3}  & \sqrt{3} \\
1 & 1 & 2  & -\sqrt{3}  & -\sqrt{3} \\
2 & 2 & -2 & 0 & 0 \\
\sqrt{3} & -\sqrt{3} & 0  & \mp \sqrt{3}  & \pm \sqrt{3} \\
\sqrt{3} & -\sqrt{3} & 0 & \pm \sqrt{3} & \mp \sqrt{3}
\end{pmatrix} \\
T &=& \mbox{diag}\left\{ 1, 1, \frac{-1+i\sqrt{3}}{2}, \frac{1+i}{\sqrt{2}}, \frac{-1-i}{\sqrt{2}} \right\}
\quad \mbox{ when } \tau =\frac{1}{\sqrt{3}}, \\
T &=&  \mbox{diag}\left\{ 1, 1, \frac{-1+i\sqrt{3}}{2}, \frac{1-i}{\sqrt{2}}, \frac{-1+i}{\sqrt{2}} \right\}
\quad \mbox{ when } \tau =-\frac{1}{\sqrt{3}}.
\end{eqnarray*}
The central charge of both the corresponding modular categories is $i$. \\
\item[(b)] When $p=3$ and $a=2$ we have:
\begin{eqnarray*}
S &=& \text{ the $S$-matrix in (a)} \\
T &=& \mbox{diag}\left\{ 1, 1, \frac{-1-i\sqrt{3}}{2}, \frac{1-i}{\sqrt{2}}, \frac{-1+i}{\sqrt{2}}  \right\}
\quad \mbox{ when } \tau =\frac{1}{\sqrt{3}}, \\
T &=& \mbox{diag}\left\{ 1, 1, \frac{-1-i\sqrt{3}}{2}, \frac{1+i}{\sqrt{2}}, \frac{-1-i}{\sqrt{2}} \right\}
\quad \mbox{ when } \tau =\frac{1}{\sqrt{3}}.
\end{eqnarray*}
The central charge of both the corresponding modular categories is $-i$.\\
\item[(c)] When $p=5$ and $a=1$ we have:
\begin{eqnarray*}
S &=&
\begin{pmatrix}
1 & 1 & 2  & 2 & \sqrt{5} & \sqrt{5} \\
1 & 1 & 2  & 2 & -\sqrt{5}  & -\sqrt{5} \\
2 & 2 & \sqrt{5}-1 & -\sqrt{5}-1 & 0 & 0 \\
2 & 2 & -\sqrt{5}-1 & \sqrt{5}-1 & 0 & 0 \\
\sqrt{5} & -\sqrt{5} & 0  & 0 & \mp \sqrt{5} & \pm \sqrt{5}\\
\sqrt{5} & -\sqrt{5} & 0  & 0 & \pm \sqrt{5} & \mp \sqrt{5}
\end{pmatrix} \\
T &=& \mbox{diag}\left\{ 1, 1, \xi^3, \xi^2, 1, -1 \right\}
\quad \mbox{ when } \tau =\frac{1}{\sqrt{5}}, \\
T &=&  \mbox{diag}\left\{ 1, 1, \xi^3, \xi^2, i, -i \right\}
\quad \mbox{ when } \tau =-\frac{1}{\sqrt{5}}.
\end{eqnarray*}
The central charge of both the corresponding modular categories is $-1$.\\
\item[(d)] When $p=5$ and $a=2$ we have:
\begin{eqnarray*}
S &=&
\begin{pmatrix}
1 & 1 & 2  & 2 & \sqrt{5} & \sqrt{5} \\
1 & 1 & 2  & 2 & -\sqrt{5}  & -\sqrt{5} \\
2 & 2 & -\sqrt{5}-1 & \sqrt{5}-1 & 0 & 0 \\
2 & 2 & \sqrt{5}-1 & -\sqrt{5}-1 & 0 & 0 \\
\sqrt{5} & -\sqrt{5} & 0  & 0 & \mp \sqrt{5} & \pm \sqrt{5}\\
\sqrt{5} & -\sqrt{5} & 0  & 0 & \pm \sqrt{5} & \mp \sqrt{5}
\end{pmatrix} \\
T &=& \mbox{diag}\left\{ 1, 1, \xi, \xi^4, i, -i \right\}
\quad \mbox{ when } \tau =\frac{1}{\sqrt{5}}, \\
T &=&  \mbox{diag}\left\{ 1, 1, \xi, \xi^4, 1, -1 \right\}
\quad \mbox{ when } \tau =-\frac{1}{\sqrt{5}}.
\end{eqnarray*}
The central charge of both the corresponding modular categories is $1$.

\end{enumerate}

\section{Appendix: Zeroes in $S$-matrices}
\label{sect 6}
There is a  classical result  of Burnside
 in character theory saying that if $\chi$ is an irreducible character of
a finite group $G$ and $\chi(1) > 1$ then  $\chi(g)=0$ for some $g\in G$,
see \cite[Chapter 21]{BZ}.

In this appendix we establish a categorical analogue of this result
for weakly integral modular categories. Recall \cite{ENO2} that a fusion category $\C$ is
called {\em weakly integral} if its Frobenius-Perron dimension is an integer.
In this case  the Frobenius-Perron dimension of every simple object of $\C$
is the square root of an integer \cite{ENO1}.

Let $\C$ be a weakly integral modular category with the $S$-matrix $S$.
Let $\O(\C)$ denote the set of all (representatives of isomorphism
classes of) simple object of $\C$. Given $X\in \O(\C)$
define the following sets:
\begin{eqnarray*}
T_X &=& \{ Y \in \O(\C) \mid  S_{X,Y}=0 \},\\
D_X &=& \O(\C) - (T_X \cup \{ \be \}).
\end{eqnarray*}
Clearly, we have a partition $\O(\C) = T_X \cup D_X  \cup \{\be \}$.
Let $\T_X$ and $\D_X$ be full Abelian subcategories of $\C$ generated
by $T_X$ and $D_X$, respectively.

Let $K$ be the field extension of $\mathbb{Q}$ generated by the entries of $S$.
It is known  \cite{dBG, CG} that there is a root of unity $\xi$ such that
$K \subset \mathbb{Q}(\xi)$.  In particular,  the operation of taking the square of an absolute
value of an element of $S$ is well defined.
Let $G := \mbox{Gal}(K/\mathbb{Q})$. Every element $\sigma\in G$ comes from a permutation
$\sigma$ of $\O(\C)$ such that $\sigma(S_{X,Y}) = S_{X,\sigma(Y)}$ for all $X,Y\in \O(\C)$.

Let $\C$ be a weakly integral modular category.
It was shown in \cite{ENO1} that there is a canonical spherical structure on $\C$
such that categorical dimensions in $\C$ coincide with Frobenius-Perron
dimensions. Let us fix this structure for the reminder of this section.
For any $X\in \O(\C)$ let $d_X$ denote the dimension of $X$.
For any full abelian subcategory $\mathcal{A}$ of $\C$ let $\dim(\mathcal{A})$
denote the sum of squares of dimensions of simple objects of $\mathcal{A}$.

\begin{theorem}
\label{Burnside again}
Let $\C$ be a weakly integral modular category with the $S$-matrix $S$.
Then $T_X$ is not empty for every non-invertible simple object $X$ of $\C$.
That is, every row (column) of $S$ corresponding to a non-invertible
simple object contains  at least one zero entry.
\end{theorem}
\begin{proof}
Note that the statement of Proposition does not depend on the choice of
spherical structure.

We have $\sum_{Y\in \O(\C)}\, |S_{X,Y}|^2 =\dim(\C)$, hence,
\begin{equation}
\label{one}
1 =\frac{\dim(\C)}{d_X^2} - \sum_{Y \in D_X}\, \left| \frac{S_{X,Y}}{d_X}\right|^2
= \frac{1+ \dim(\T_X)}{d_X^2} - \left(  \sum_{Y \in D_X}\, \left| \frac{S_{X,Y}}{d_X}\right|^2
-  \frac{\dim(\D_X)}{d_X^2}\right),
\end{equation}
where $d_X$ denotes the dimension of $X$. It suffices to check that
\begin{equation}
\label{two}
\frac{1}{\dim(\D_X)}  \sum_{Y \in D_X}\, \left| \frac{S_{X,Y}}{d_X}\right|^2 \geq \frac{1}{d_X^2}
\end{equation}
since then  \eqref{one} implies that $ 1\leq \tfrac{1+ \dim(\T_X)}{d_X^2}$,
whence
\begin{equation}
\label{three}
\dim(\T_X) \geq d_X^2-1.
\end{equation}
But $X$ is non-invertible so $d_X > 1$ and $\T_X \neq 0$.

Rewriting the left hand side of \eqref{two} as the sum  of $\dim(\D_X)$ terms and
using the inequality of arithmetic and geometric means we obtain
\begin{eqnarray*}
\frac{1}{\dim(\D_X)}  \sum_{Y \in D_X}\, \left| \frac{S_{X,Y}}{d_X}\right|^2
&=& \frac{1}{\dim(\D_X)}  \sum_{Y \in D_X}\, d_Y^2 \left| \frac{S_{X,Y}}{d_Xd_Y}\right|^2 \\
&\geq& \frac{1}{d_X^2} \left(   \prod_{Y \in D_X}\, \left| \frac{S_{X,Y}}{d_Y}\right|^{2d_Y^2}
  \right)^{\frac{1}{\dim(D_X)}}.
\end{eqnarray*}

The set $D_X$ is clearly stable under  all automorphisms in the Galois group, and hence so is
the product $\prod_{Y \in D_X}\, \left| \frac{S_{X,Y}}{d_Y}\right|^{2d_Y^2}$.
Therefore, this product belongs to $\mathbb{Q}$. Its factors
are squares of absolute values of characters of $K_0(\C)$ on $X$ and hence are
algebraic integers.  Since all factors are positive, the product is  $\geq 1$,
which implies \eqref{two}.
\end{proof}


For $X\in \O(\C)$ define
\[
U_X =\{ Y \in \O(\C) \mid  |S_{X,Y}| =d_Y \}.
\]
Let $\mathcal{U}_X$ be the full Abelian subcategory of $\C$ generated by $U_X$.

\begin{proposition}
\label{Thompson}
Let $\C$ be a weakly integral modular category  and let
$X$ be a simple non-invertible object in $\C$. Then
\begin{equation}
3\dim(\T_X) +\dim(\mathcal{U}_X) > \dim(\C).
\end{equation}
\end{proposition}
\begin{proof}
We may assume $d_X \geq \sqrt{2}$.

We will use the following theorem of Siegel \cite{Si} from number theory. Let $K/\mathbb{Q}$
be a finite Galois extension with the Galois group $G =\mbox{Gal}(K/\mathbb{Q})$.
Let $\alpha$ be a totally positive algebraic integer in $K$, $\alpha \neq 1$. Then
\[
\frac{1}{|G|} \sum_{\sigma\in G}\, \sigma(\alpha) \geq \frac{3}{2}.
\]
We apply this to the situation when $K$ is the extension of $\mathbb{Q}$
generated by entries of $S$. We compute
\begin{eqnarray*}
\dim(\C)
&=&  \sum_{Y\in \C}\, |S_{X,Y}|^2 \\
&=&  d_X^2 +  \sum_{Y\in U_X}\, d_Y^2 + \sum_{Y\in \O(\C) -(T_X \cup U_X\cup \{\be\})}\,|S_{X,Y}|^2 \\
&=&  d_X^2 + \dim(\mathcal{U}_X) +
     \sum_{Y\in \O(\C) -(T_X \cup U_X\cup \{\be\})}\, d_Y^2
     \left(  \frac{1}{|G|} \sum_{\sigma \in G}\, \sigma\left( \frac{|S_{X,Y}|^2}{d_Y^2} \right) \right) \\
&\geq& 2 + \dim(\mathcal{U}_X) +  \frac{3}{2} (\dim(\C)-\dim(\T_X)-\dim(\mathcal{U}_X)-1),
\end{eqnarray*}
therefore $3 \dim(\T_X) +\dim(\mathcal{U}_X) \geq \dim(\C)+ 1 > \dim(\C)$, as required.
\end{proof}

\begin{remark}
Our proofs of Theorem~\ref{Burnside again} and Proposition~\ref{Thompson} imitate the corresponding
proofs for group characters given in \cite{BZ}.
\end{remark}

\bibliographystyle{ams-alpha}


\end{document}